\documentclass[11pt]{article}
\usepackage{amsthm, amsmath, amssymb, amsfonts, url, booktabs, tikz, setspace, fancyhdr}
\usepackage[hidelinks]{hyperref}
\usepackage[margin = 1in]{geometry}
\usepackage{hyperref, enumerate}
\usepackage[shortlabels]{enumitem}
\usepackage[babel]{microtype}
\usepackage[english]{babel}
\usepackage{cleveref}
\usepackage{comment}
\usetikzlibrary{arrows.meta}

\usepackage{enumitem}
\newlist{steps}{enumerate}{1}
\setlist[steps, 1]{label = Step \arabic*:}

\newtheorem{theorem}{Theorem}[section]
\newtheorem{proposition}[theorem]{Proposition}
\newtheorem{lemma}[theorem]{Lemma}
\newtheorem{corollary}[theorem]{Corollary}
\newtheorem{conjecture}[theorem]{Conjecture}

\newtheorem{question}[theorem]{Question}

\theoremstyle{definition}

\newtheorem{remark}[theorem]{Remark}

\theoremstyle{remark}

\numberwithin{equation}{section}

\crefname{equation}{}{}
\Crefname{equation}{Equation}{Equations}
\crefname{theorem}{Theorem}{Theorems}
\Crefname{theorem}{Theorem}{Theorems}
\crefname{lemma}{Lemma}{Lemmas}
\Crefname{lemma}{Lemma}{Lemmas}
\crefname{proposition}{Proposition}{Propositions}
\Crefname{proposition}{Proposition}{Propositions}
\crefname{corollary}{Corollary}{Corollaries}
\Crefname{corollary}{Corollary}{Corollaries}
\crefname{conjecture}{Conjecture}{Conjectures}
\Crefname{conjecture}{Conjecture}{Conjectures}
\crefname{section}{Section}{Sections}
\Crefname{section}{Section}{Sections}
\crefname{example}{Example}{Examples}
\Crefname{example}{Example}{Examples}
\crefname{problem}{Problem}{Problems}
\Crefname{problem}{Problem}{Problems}
\crefname{remark}{Remark}{Remarks}
\Crefname{remark}{Remark}{Remarks}
\crefname{figure}{Figure}{Figures}
\Crefname{figure}{Figure}{Figures}

\newcommand{\overbar}[1]{\mkern 1.5mu\overline{\mkern-1.5mu#1\mkern-1.5mu}\mkern 1.5mu}

\newcommand{\G}{\mathcal{G}}

\newcommand{\RR}{\mathbb{R}}
\newcommand{\CC}{\mathbb{C}}
\newcommand{\ZZ}{\mathbb{Z}}

\newcommand{\FF}{\mathbb{F}}

\newcommand{\LL}{\mathcal{L}}

\newcommand{\A}{\mathcal{A}}
\newcommand{\B}{\mathcal{B}}
\newcommand{\C}{\mathcal{C}}
\newcommand{\D}{\mathcal{D}}

\newcommand{\F}{\mathcal{F}}

\newcommand{\oo}{\mathcal{O}}

\newcommand{\sym}{\mathrm{\, sym}}

\usetikzlibrary{trees,decorations.markings}
\tikzstyle{p}+=[fill=black, circle, minimum width = 1pt, inner sep =
1pt]
\tikzstyle{w}+=[fill=white, draw, circle, minimum width = 1pt, inner sep =
1.5pt]

\tikzset{->-/.style={decoration={
  markings,
  mark=at position #1 with {\arrow{>}}},postaction={decorate}}}
  
\tikzset{-<-/.style={decoration={
  markings,
  mark=at position #1 with {\arrow{<}}},postaction={decorate}}}

\begin{document}

\title{On graph norms for complex-valued functions}
\author{
Joonkyung Lee\thanks{
Department of Mathematics, Hanyang University, Seoul, South Korea,
E-mail: {\tt
joonkyunglee@hanyang.ac.kr}.} 
\and
Alexander Sidorenko\thanks{
R\'{e}nyi Institute, Budapest, Hungary,
E-mail: {\tt
sidorenko.ny@gmail.com}
} 
}
\maketitle

\begin{abstract}
    For any given graph $H$, one may define a natural corresponding functional $\|.\|_H$ for real-valued functions by using homomorphism density. One may also extend this to complex-valued functions, once $H$ is paired with a $2$-edge-colouring $\alpha$ to assign conjugates.
    We say that $H$ is \emph{real-norming} (resp. \emph{complex-norming}) if $\|.\|_H$ (resp. $\|.\|_{H,\alpha}$ for some $\alpha$) is a norm on the vector space of real-valued (resp. complex-valued) functions. 
    These generalise the Gowers octahedral norms, a widely used tool in extremal combinatorics to quantify quasirandomness.
    
    We unify these two seemingly different notions of graph norms in real- and complex-valued settings.
    Namely, we prove that $H$ is complex-norming if and only if it is real-norming and simply call the property \emph{norming}.
    Our proof does not explicitly construct a suitable $2$-edge-colouring $\alpha$ but obtains its existence and uniqueness, which may be of independent interest.
    
    As an application, we give various example graphs that are not norming. 
    In particular, we show that hypercubes are not norming, which 
    resolves the last outstanding problem posed in Hatami's pioneering work on graph norms.   
\end{abstract}

\section{Introduction}
One of the key concepts in the theory of quasirandomness is the \emph{Gowers octahedral norms}, introduced by Gowers~\cite{G06,G07} for his proof of hypergraph regularity lemma.
In its full generality, the Gowers octahedral norms can be defined for complex-valued functions $f:[0,1]^2\rightarrow \mathbb{C}$. The simplest example is the $C_4$-norm, written as
\begin{align}\label{eq:C4norm}
   \|f\| :=
   \left\vert\int f(x_1,y_1)\overbar{f(x_1,y_2)}\overbar{f(x_2,y_1)}f(x_2,y_2) dx_1dy_1dx_2dy_2\right\vert^{1/4}.
\end{align}
This is in fact a graph-theoretic interpretation of \emph{Gowers uniformity norms} in additive combinatorics, a crucial tool in Gowers' proof~\cite{G01} of celebrated Szemer\'edi's theorem.
In additive combinatorial contexts, complex-valued functions are naturally considered because of the connections to the Fourier transform.

However, generalisations of the $C_4$-norm have been mostly studied in a real-valued setting. 
Let 
\begin{align*}
     t_H(f) :=  
  \int \prod_{uv \in E(H)}
    f(x_u,y_v) d\mu^{|V(H)|} ,
\end{align*}
where $\mu$ denotes the Lebesgue measure on $[0,1]^2$.
Lov\'asz~\cite{L08} asked  
to characterise {\it real-norming} graphs $H$, such that $\|f\|_H:=|t_H(f)|^{1/{\rm e}(H)}$
defines a norm on the vector space of bounded measurable functions $f:[0,1]^2\rightarrow \mathbb{R}$. 
One may also ask an analogous question by replacing the functions $f$ by its absolute value $|f|$. If this replacement defines a norm, then the graph $H$ is called \emph{weakly norming}. Indeed, as Hatami~\cite{H10} observed, this is a weaker property than real-norming. 

It was Hatami~\cite{H09} who first raised a complex-valued analogue of the Lov\'asz question and 
extended in this direction some results
of his pioneering work~\cite{H10} in the area of graph norms.
Conlon and Lee~\cite{CL16} briefly discussed in the concluding remarks that the class of real-norming graphs they obtained is also norming in the complex-valued sense.
To the best of our knowledge, these are the only literature that considered graph norms for complex-valued functions.

To describe Hatami's complex-valued analogue of the Lov\'asz question, 
note first that the weakly norming property remains the same for the complex-valued functions~$f$, as $|f(x,y)|$ is always a real-valued function. Thus, what remains interesting is an analogue of the real-norming property for complex-valued functions.
To this end, one should pair a graph with its $2$-edge-colouring to assign conjugates, as was done in~\eqref{eq:C4norm} and more generally in~\cite{H09}. 
Our main object throughout this paper is hence a pair $(H,\alpha)$ of a bipartite graph\footnote{A real-norming graph is necessarily bipartite, as was observed by Hatami~\cite{H10}.} $H$ and a $2$-edge-colouring $\alpha : \: E(H) \to \{0,1\}$.
We call this pair a \emph{$2$-coloured graph}.
For brevity, let $\F_{\CC}$ be the class of all bounded measurable functions 
$f:\: [0,1]^2 \to \CC$. 
For $f\in\F_{\CC}$, define 
\[
  t_{H,\alpha}(f) : = \: 
  \int \prod_{e=(u,v) \in E(H)}
    f(x_u,y_v)^{\alpha(e)} \:
    \overline{f(x_u,y_v)}^{1-\alpha(e)} \; d\mu^{|V(H)|} .
\]
For example, in~\eqref{eq:C4norm}, $H=C_4$, $\alpha$ alternates colours of edges along the $4$-cycle, and $\|f\|=|t_{C_4,\alpha}(f)|^{1/4}$.

We say that a $2$-coloured graph $(H,\alpha)$ is {\it complex-norming} if $\|f\|_{H,\alpha}:=|t_{H,\alpha}(f)|^{1/{\rm e}(H)}$ 
is a norm on $\F_{\CC}$. A graph $H$ is also called complex-norming if there exists a $2$-edge-colouring $\alpha$ of $H$ that makes $(H,\alpha)$ complex-norming.
For real-valued $f$, the value $t_{H,\alpha}(f)$ does not depend on~$\alpha$. Therefore, $(H,\alpha)$ (or $H$) being complex-norming is a stronger property than $H$ being real-norming.
We also remark that our definition slightly differs from that of Hatami \cite{H09}, where
$t_{H,\alpha}(f)$ are assumed to be real and nonnegative and then $t_{H,\alpha}(f)^{1/{\rm e}(H)}$ is required to be a norm. 
Although we begin with this less restrictive requirement, both definitions will be shown to be equivalent by \cref{th:no_abs_value}.
Our main result states that the two seemingly different properties, real-norming and complex-norming, are in fact equivalent.
\begin{theorem}\label{thm:equiv}
A graph $H$ is complex-norming if and only if it is real-norming.
\end{theorem}
We also obtain uniqueness of a $2$-edge-colouring that makes $H$ complex-norming. The precise definition of isomorphism will be given in~\cref{sec:characterisation}.
\begin{theorem}\label{thm:unique}
Let $H$ be a complex-norming graph. Then a $2$-edge-colouring $\alpha$ such that $(H,\alpha)$ is complex-norming is unique up to isomorphism.
\end{theorem}
One might imagine that, to prove~\cref{thm:equiv}, a suitable $2$-edge-colouring $\alpha$ should be constructed for every real-norming $H$.
To illustrate the complexity of this problem, we give an example of a $2$-coloured graph $(H,\alpha)$ in \cref{fig:octahedron}; 
it is already not straightforward to construct such a suitable colouring $\alpha$ for the 1-subdivision of an octahedron, known to be real-norming by~\cite[Example~4.15]{CL16}, unless one uses the algebraic information discussed therein. Furthermore, it has been unknown whether it is a unique $2$-edge-colouring that makes $(H,\alpha)$ complex-norming.

However, our proof is surprisingly indirect in such a way that we do not construct any $2$-edge-colouring~$\alpha$ but prove its existence and uniqueness.
\Cref{thm:equiv} then unifies Hatami's analogue to the original Lov\'asz question about real-valued functions. As a consequence, except in \cref{sec:characterisation,sec:blind}, where we prove \cref{thm:equiv}, we shall simply say that a graph is \emph{norming} if it is real-norming (and hence, complex-norming too).

Our proof technique also obtains another interesting equivalence statement. 
In the theory of graph limits, it is more common to consider \emph{symmetric} functions $f:[0,1]^2\rightarrow\mathbb{R}$, i.e., $f(x,y)=f(y,x)$ for every $(x,y)\in[0,1]^2$. In particular, if the range of $f$ is in $[0,1]$, then $f$ is a \emph{graphon}, which appears as the natural limit object of graphs~\cite{L12,LSz06}. Indeed, one may restrict the definition to symmetric (or Hermitian, in the complex-valued case) functions and study graph norms in the restrictive setting. This is precisely what Lov\'asz~\cite{L12} did, whereas Hatami~\cite{H10} considered possibly asymmetric functions.
An analogous but somewhat simpler approach to the proof of \cref{thm:equiv} proves that this symmetric setting taken by Lov\'asz is essentially equivalent to the asymmetric one. 

\begin{theorem}\label{thm:equivalence_asymm}
Let $H$ be a connected bipartite graph. 
If $H$ is real-norming (resp. weakly norming) for symmetric functions, 
then $H$ is real-norming (resp. weakly norming) for asymmetric functions.
\end{theorem}
The connectivity requirement for $H$ is essential; for example, 
the vertex-disjoint union of $K_{2,4}$ and $K_{4,2}$ (with the equal-sized bipartition) is norming for symmetric functions but not in general for asymmetric functions.\footnote{This will easily follow from~\cref{th:edge-transitive}.}

\medskip

Theorem~\ref{thm:equiv} can also be seen as a clue to the mysterious connection between graph norms and group theory.
Let $G$ be a finite group and let $G_1$ and $G_2$ be its subgroups. The \emph{$(G_1,G_2;G)$-graph} is a bipartite graph such that each part in the bipartition is the set of all left cosets of $G_i$, $i=1,2$, and an edge exists between two cosets if and only if they intersect.
A recent result by Sidorenko~\cite{Sidorenko20} stating that every weakly norming graph is edge-transitive implies that every weakly norming graph~$H$ is a $(G_1,G_2;G)$-graph, where $G$ is the group of automorphisms of $H$ that preserves the bipartition.

On the other hand, Conlon and Lee~\cite{CL16} proved that, if $G$ is a finite reflection group and the subgroups $G_1$ and $G_2$ are \emph{parabolic}, i.e., generated by simple reflections, then the $(G_1,G_2;G)$-graph is always weakly norming. This class of $(G_1,G_2;G)$-graphs are called \emph{reflection graphs}.
They further showed~\cite[Theorem~1.3]{CL16} that if $G_1$ and $G_2$ only have a trivial intersection, then the $(G_1,G_2;G)$-graph that is not a union of stars\footnote{The functional $\|.\|_H$ is a seminorm if and only if $H$ is a disjoint union of isomorphic stars with even number of leaves (see Chapter 14 in~\cite{L12}), and thus, we exclude this case to be exact. See \cref{th:seminorm} for more relevant discussions.} is norming. For brevity, let us call these graphs \emph{stable reflection graphs}. The example shown in~\cref{fig:octahedron} is also a stable reflection graph.  
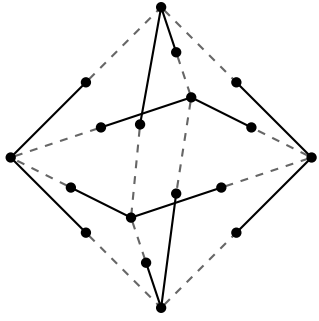
\begin{figure}
    \centering
    \begin{tikzpicture}[thick,scale=4]
\node[circle, draw, fill=black, scale=.3] (x1) at (0,0){};
\node[circle, draw, fill=black, scale=.3] (xy12) at (0.3,0.1) {};
\node[circle, draw, fill=black, scale=.3] (xy11) at (0.2,-0.1) {};
\node[circle, draw, fill=black, scale=.3] (xz11) at (0.25,0.25) {};
\node[circle, draw, fill=black, scale=.3] (xz12) at (0.25,-0.25) {};

\node[circle, draw, fill=black, scale=.3] (x2) at (1,0) {};
\node[circle, draw, fill=black, scale=.3] (xy21) at (0.7,-0.1) {};
\node[circle, draw, fill=black, scale=.3] (xy22) at (0.8,0.1) {};
\node[circle, draw, fill=black, scale=.3] (xz21) at (0.75,0.25) {};
\node[circle, draw, fill=black, scale=.3] (xz22) at (0.75,-0.25) {};

\node[circle, draw, fill=black, scale=.3] (y1) at (0.4,-0.2) {};
\node[circle, draw, fill=black, scale=.3] (yz11) at (0.43, 0.11) {};
\node[circle, draw, fill=black, scale=.3] (yz12) at (0.45, -0.35) {};

\node[circle, draw, fill=black, scale=.3] (y2) at (0.6,0.2) {};
\node[circle, draw, fill=black, scale=.3] (yz21) at (0.55,0.35) {};
\node[circle, draw, fill=black, scale=.3] (yz22) at (0.55,-0.12) {};

\node[circle, draw, fill=black, scale=.3] (z1) at (0.5,0.5) {};
\node[circle, draw, fill=black, scale=.3] (z2) at (0.5,-0.5) {};

\begin{scope}[thick,dashed,,opacity=0.6]
\draw (xy11) -- (x1) -- (xy12);
\draw (xy21) -- (x2) -- (xy22);
\draw (xz11) -- (z1) -- (xz21);
\draw (yz11) -- (y1) -- (yz12);
\draw (xz12) -- (z2) -- (xz22);
\draw (yz21) -- (y2) -- (yz22);
\end{scope}
\draw (yz21) -- (z1) -- (yz11);
\draw (xz11) -- (x1) -- (xz12);
\draw (xz21) -- (x2) -- (xz22);
\draw (xy11) -- (y1) -- (xy21);
\draw (xy12) -- (y2) -- (xy22);
\draw (yz12) -- (z2) -- (yz22);

\end{tikzpicture}
    \caption{The 1-subdivision of an octahedron is complex-norming.}
    \label{fig:octahedron}
\end{figure}
Since the class of stable reflection graphs include all the known examples of norming graphs, Conlon~\cite{C20} conjectured the following: 
\begin{conjecture}[Conlon]\label{conj:norming}
Every norming graph is a stable reflection graph.
\end{conjecture}
In other words, the conjecture states that a graph is norming if and only if it is isomorphic to a special type of $(G_1,G_2;G)$-graphs, where $G$ is a finite reflection graph and $G_i$, $i=1,2$ are parabolic subgroups with $|G_1\cap G_2|=1$.
This gives a surprising correspondence between a functional-analytic property and a group-theoretic description, if true.

In stable reflection graphs, each reflection $w\in G$ acts as an automorphism with extremely strong property, so-called a \emph{stable involution}. That is, the vertex set of the $(G_1,G_2;G)$-graph can be partitioned into $L\cup R\cup F$, where $F$ is the set of fixed vertices that induces a vertex cut with no edges inside and $L$ and $R$ are mapped to each other while no edges cross between the two. In particular, $w$ acts as an involution on the edge set and the vertex set while fixing no edges. Thus, \cref{conj:norming} not only suggests the complete answer to the Lov\'asz question but also implies Conjecture~6.3 in~\cite{CL16}, which states that a graph is norming if and only if it is edge-transitive under the automorphism subgroup generated by stable involutions.

Furthermore, the existence of a single stable involution, a weaker statement than \cref{conj:norming} (or even Conjecture 6.3 in~\cite{CL16}), is already enough to settle some special cases of well-known conjectures.
For example, the positive graph conjecture~\cite{CCHLL12} states that if $t_H(f)\geq 0$ for all symmetric $f:[0,1]^2\rightarrow\mathbb{R}$, then $H$ has a stable involution. For instance, as observed in~\cite{L12}, every norming graph $H$ satisfies $t_H(f)\geq 0$ for every symmetric $f:[0,1]^2\rightarrow \mathbb{R}$. 
Hence, the existence of stable involutions in norming graphs implies the positive graph conjecture for norming graphs, asked precisely in~\cite{CCHLL12} as a particularly interesting case. 

\cref{conj:norming} also connects to the Polycirculant conjecture \cite{Marusic81}, 
which states that every vertex-transitive graph has a non-identity automorphism that has a uniform length for every vertex orbit. Since every norming graph is edge-transitive as proven in~\cite{Sidorenko20}, the existence of a stable involution will prove the Polycirculant conjecture for line graphs of norming graphs in a strong sense, which adds plenty of examples to the conjecture.  

In~\cite{CL16}, a $2$-edge-colouring $\alpha$ that makes $(H,\alpha)$ norming for a stable reflection graph $H$ is obtained by a particular index-two subgroup of the underlying reflection group, so-called the \emph{alternating subgroup} or the \emph{rotation subgroup}.
In particular, Conjecture~\ref{conj:norming} implies Theorem~\ref{thm:equiv}, so the theorem can be seen as a positive evidence for the conjecture.
Furthermore, \cref{thm:equiv} together with \cref{th:edge-transitive} will also prove that there exists an automorphism subgroup of a norming graph $H$ that acts transitively on the edges and has an index-two subgroup.
This provides another nontrivial evidence that supports Conjecture~\ref{conj:norming}.

\medskip

In~\cite{H10}, Hatami proposed six open problems about graph norms, all but one of which have been at least partly answered 
in \cite{CL16,GHL19,H09,KMPW19,LS19}. 
As an application of \cref{thm:equiv}, 
we answer the last remaining open question which asks whether hypercubes, proven to be weakly norming by Hatami, are norming. 
\begin{theorem}\label{thm:cubes}
The $d$-dimensional hypercube $Q_d$ is not norming whenever $d>2$.
\end{theorem}
We also show that various graphs that are proven to be weakly norming in~\cite{CL16} are not norming, which provides a plenty of examples that support Conjecture~\ref{conj:norming} and indicates that \cref{thm:equiv} is so far the most powerful tool to distinguish weakly norming graphs and norming graphs.

\medskip

The paper is organised as follows. 
In \cref{sec:prelim}, we prove some lemmas that will play important roles in the main proofs.
\cref{thm:equivalence_asymm} and \cref{thm:equiv} will be proved in~\cref{sec:symmetry} and~\cref{sec:blind}, respectively, while in~\cref{sec:characterisation}, some fundamental properties of complex-norming graphs will be obtained. 
We then apply these properties and~\cref{thm:equiv} in~\cref{sec:cycles,sec:set-inclusion} to show that some weakly norming graphs obtained in~\cite{CL16} and~\cite{H10} are not norming.
In particular,~\cref{thm:cubes} will be proved in Subsection~\ref{sec:cubes} and the bipartite Kneser graphs will be studied in~\cref{sec:set-inclusion}.

\section{Analytic preliminaries}\label{sec:prelim}

Throughout the paper, let $[k]:=\{1,2,\ldots,k\}$. Recall that $\F_\CC$ denotes the class of all bounded measurable functions $f:\: [0,1]^2 \to \CC$. We analogously define $\F_{\RR}$ by replacing $\CC$ by the real field $\RR$. 
Throughout this section, let $\F$ be either $\F_{\RR}$ or $\F_{\CC}$. The \emph{underlying field} $\mathbb{F}$ denotes either $\RR$ or~$\CC$ that corresponds the choice of $\F$.

We briefly say $f=g$ \emph{a.e.} 
if $f$ is equal to $g$ almost everywhere. 
The \emph{tensor product} $f\otimes g$ of $f,g\in \F$ 
is defined by the function 
$f\otimes g ((x,y),(z,w)):=f(x,z)g(y,w)$ on $[0,1]^4$. 
Since there is a measure-preserving bijection 
from $[0,1]^2$ onto $[0,1]$, 
we always regard $f\otimes g$ as a function on $[0,1]^2$ 
to preserve the consistency of the domain measure space.
We write 
$f^{\otimes n}:=f\otimes f\otimes \cdots \otimes f$ 
for~$n$ iterative tensor powers of $f$.

In what follows, it is crucial to define a functional $\tau$ that shares many properties with graph homomorphism densities $t_H(\cdot)$ or $t_{H,\alpha}(\cdot)$ and to argue that it also defines a (semi-)norm. 
To this end, we axiomatise the properties we need and obtain three lemmas that yield some functional-analytic consequences.
Firstly, a \emph{decoration functional} on $\F^k$ is a function from $\F^k$ to $\mathbb{F}$ that satisfies the following conditions:
\begin{enumerate}[(i)]
\item\label{it:scalar} $|\tau(cf_1,\ldots,cf_k)|=|c|^k|\tau(f_1,\ldots,f_k)|$  for each $c\in \mathbb{F}$;
\item\label{it:linear} $\tau(f_1,\ldots,f_{i-1},g+ah,f_{i+1},\ldots,f_k)$\\$=
  \tau(f_1,\ldots,f_{i-1},g,f_{i+1},\ldots,f_k)  + 
   a\tau(f_1,\ldots,f_{i-1},h,f_{i+1},\ldots,f_k)$
   
   for any $i=1,2,\ldots,k$, any $f_j\in\F$ ($j \neq i$), \:$g,h\in\F$ and a \emph{real} number $a$;
\item\label{it:conjugate}
$\tau(\overline{f}_1,\ldots,\overline{f}_k) = \overline{\tau(f_1,\ldots,f_k)}$;
\item\label{it:tensor}
$\tau(f_1 \otimes g_1,\ldots,f_k \otimes g_k) =
\tau(f_1,\ldots,f_k) \, \tau(g_1,\ldots,g_k)$.
\end{enumerate}
Some functionals $\tau$ we shall use do not satisfy~\ref{it:linear} in general but possess a weaker `subadditivity':
\begin{enumerate}[(i')]
    \setcounter{enumi}{1}
    \item\label{it:subadd} $|\tau(f_1,\ldots,f_{i-1},g+h,f_{i+1},\ldots,f_k)|$\\$\leq
  |\tau(f_1,\ldots,f_{i-1},g,f_{i+1},\ldots,f_k)|  + 
   |\tau(f_1,\ldots,f_{i-1},h,f_{i+1},\ldots,f_k)|$ 
   for any $i=1,2,\ldots,k$.
\end{enumerate}
If $\tau$ satisfies \ref{it:scalar} and \ref{it:subadd}, then we say that $\tau$ is a \emph{weak decoration functional}.

For a (weak) decoration functional $\tau:\F^k\rightarrow\mathbb{F}$, we abuse the notation slightly by writing  $\tau(f)=\tau(f,\ldots,f)$. 
To clarify, we call this $\tau:\F\rightarrow\mathbb{F}$ a \emph{density functional} if its corresponding multilinear functional $\tau(f_1,\ldots,f_k)$ is a decoration functional, not a weak one.

\medskip

The first lemma is a generalisation of 
Hatami's inequality for graph norms \cite[Theorem~2.8]{H10}.

\begin{lemma}\label{th:Hatami_gen}
Let $\tau$ be a weak decoration functional on $\F^k$.
Then $|\tau(f)|^{1/k}$ is a seminorm on $\F$ if
for any $f_1,\ldots,f_k \in \F$,
\begin{equation}\label{eq:Hatami_gen}
  |\tau(f_1,\ldots,f_k)|^k \; \leq \prod_{i=1}^k |\tau(f_i)| \, .
\end{equation}
The converse also holds if $\tau$ is a decoration functional.
\end{lemma}

\begin{proof}[\bf{Proof}]
By the conditions~\ref{it:scalar} and~\ref{it:subadd}, $\tau(0)=0$ and $|\tau(cf)|^{1/k} = |c||\tau(f)|^{1/k}$.
Thus, it suffices to prove the triangle inequality.
Let $f_1,f_2\in\F$ and let $\Phi$ be the collection 
of all maps $\phi: \{1,2,\ldots,k\}\rightarrow\{1,2\}$. 
By using~\ref{it:subadd} and~\eqref{eq:Hatami_gen}, we obtain
\begin{align*}
  |\tau(f_1+f_2)| \:\leq \:
  \sum_{\phi\in\Phi} \left|\tau\big(f_{\phi(1)},\ldots,f_{\phi(k)}\big)\right| \: \leq \:
  \sum_{\phi\in\Phi}\prod_{i=1}^k \left|\tau\big(f_{\phi(i)}\big)\right|^{1/k} = 
  \left(|\tau(f_1)|^{1/k} + |\tau(f_2)|^{1/k}\right)^k,
\end{align*}
which proves the triangle inequality.

\medskip

Conversely, suppose that $\tau$ is a decoration functional and the triangle inequality holds for 
$|\tau(\cdot)|^{1/k}$. 
By homogeneity of~\eqref{eq:Hatami_gen}, 
we may assume $|\tau(f_i)|=1$ and prove 
$|\tau(f_1,\ldots,f_k)| \leq 1$.
For an arbitrary $n$, set $f:=\sum_{i=1}^k f_i^{\otimes n}\otimes \overline{f_i}^{\otimes n}$. 
Then by multilinear expansion,
\begin{align*}
   \tau(f)&= \sum_{\psi:[k]\rightarrow [k]} 
  \tau\Big(
    f_{\psi(1)}^{\otimes n}\otimes \overline{f}_{\psi(1)}^{\otimes n},
    \ldots,
    f_{\psi(k)}^{\otimes n}\otimes \overline{f}_{\psi(k)}^{\otimes n}
  \Big)\\
  &= \sum_{\psi:[k]\rightarrow [k]} 
  \tau\big(f_{\psi(1)},\ldots,f_{\psi(k)}\big)^n \:
  \overline{\tau\big(f_{\psi(1)},\ldots,f_{\psi(k)}\big)^n}\\
  &= \sum_{\psi:[k]\rightarrow [k]} \big|\tau\big(f_{\psi(1)},\ldots,f_{\psi(k)}\big)\big|^{2n}.
\end{align*}
In particular, $\tau(f)$ is always real. Therefore,
\begin{align*}
|\tau(f_1,\ldots,f_k)|^{2n}
&\leq \sum_{\psi:[k]\rightarrow [k]} \big|\tau\big(f_{\psi(1)},\ldots,f_{\psi(k)}\big)\big|^{2n} 
=\tau(f) = |\tau(f)| 
\leq \left(\sum_{i=1}^k |\tau(f_i)|^{2n/k}\right)^k \!= k^k,
\end{align*}
where the first inequality follows from nonnegativity of 
$|\tau(f_{\psi(1)},\ldots,f_{\psi(k)})|$ 
and the last is the triangle inequality together with \ref{it:conjugate} and \ref{it:tensor}.
Since $|\tau(f_1,\ldots,f_k)| \leq c^{1/2n}$ 
for a constant $c>1$ and \emph{any} positive integers $n$, we get 
$|\tau(f_1,\ldots,f_k)| \leq 1$.
\end{proof}

The next lemma generalises an analogous result for graph norms \cite[Theorem~1.5~(ii)]{LS19} to an abstract setting by following the proof outline of \cite[Theorem~4.1]{LS19}.

\begin{lemma}\label{th:convexity_gen}
Let $\tau$ be a weak decoration functional on $\F^k$. Then $|\tau(\cdot)|^{1/k}$ is a seminorm on $\F$ if and only if $|\tau(\cdot)|$ is convex, i.e., $|\tau(\frac{1}{2}(f+g))| \leq \frac{1}{2}(|\tau(f)|+|\tau(g)|)$ 
for all $f,g\in\F$. 
\end{lemma}
\begin{proof}[\bf{Proof}]
Let $\sigma(f):=|\tau(f)|$ for brevity.
As $\sigma(f)^{1/k}$ is a seminorm and in particular a convex function, its $k$-th power $\sigma(f)$ is again convex. Indeed,
\begin{align*}
    \sigma\big((f+g)/2\big)\leq \left(\frac{1}{2}\left(\sigma(f)^{1/k}+\sigma(g)^{1/k}\right)\right)^k
    \leq \frac{1}{2}\big(\sigma(f)+\sigma(g)\big),
\end{align*}
where the first and the second inequalities follow by convexity of $\sigma(\cdot)^{1/k}$ and $x^k$, respectively.

\medskip

Conversely, if $\sigma(\cdot)$ is convex, then for every distinct $f,g\in\F$ and $0<\lambda<1$,
\begin{align*}
    \sigma(\lambda f+(1-\lambda)g) \: \leq \: \lambda \sigma(f) + (1-\lambda) \sigma(g).
\end{align*}
Therefore, the unit ball
\begin{align*}
    B := \{f\in\F : \: \sigma(f)^{1/k} \leq 1\} = \{f\in\F : \: \sigma(f) \leq 1\}
\end{align*}
is convex, which is equivalent to the triangle inequality for 
$\sigma(\cdot)^{1/k}$ (see, e.g., \cite[Lemma~2.6]{LS19}).
\end{proof}
In fact, we did not use the full definition of weak decoration functionals, but only used the homogeneity property~\ref{it:scalar} for the converse part.

\medskip

The third lemma is our novel contribution and one of the key ingredients in the proofs of both
\cref{thm:equivalence_asymm,thm:equiv}.
A set $\G$ in a vector space $\F$ is called {\it algebraically open} at $f\in\G$ if for any $g\in\F$, 
there exists $\varepsilon > 0$ such that $f+xg\in\G$ 
for each $x \in (-\varepsilon,\varepsilon)$.
We simply say that $\G$ is \emph{algebraically open} if it is algebraically open at every $f\in \G$.

Let $\tau$ be a density functional on $\F$, defined by a decoration functional on $\F^k$.
We shall repeatedly use the fact that, for $f,g,h \in \F$, $\tau(f+xg+yh)$ is a two-variable homogeneous polynomial of degree at most $k$ in \emph{real} variables $x$ and $y$, where the coefficients of $x^k$ and $y^k$ are $\tau(g)$ and $\tau(h)$, respectively. This easily follows from multilinearity \ref{it:linear} of $\tau$ as a decoration functional.

\begin{lemma}\label{th:conv_min}
Let $\tau_1,\tau_2,\ldots,\tau_m$ be nonnegative real-valued density functionals on $\F$ that satisfy: 
\begin{enumerate}[(a)]
\item\label{conv_min_1}
$\tau(f) := \max_j \tau_j(f)$  is convex on $\F$;
\item\label{conv_min_2}
For each $j\in\{1,2,\ldots,m\}$, there is $f \in \F$ such that 
$\max_{\ell\neq j} \tau_\ell(f) < \tau_j(f)$;
\item\label{conv_min_4}
For any nonempty algebraically open subset $\C\subseteq\F$, 
there is $f\in\C$ such that $\tau_j(f) > 0$ 
for all $j\in\{1,2,\ldots,m\}$. 
\end{enumerate}
Then there exists $j\in\{1,2,\ldots,m\}$ 
such that $\tau_j(\cdot)$ is convex.
\end{lemma}

\begin{proof}[\bf{Proof}]
The case $m=1$ is trivial, so assume $m \geq 2$. 
Suppose that none of $\tau_j(\cdot)$ is convex.
Let $P_{j}(x;f,g):=\tau_j(f+xg)$. 
Then $P_{j}(x;f,g)$ is a polynomial in $x$ and its second derivative $P''_{j}(x;f,g)$ exists for every $f,g\in\F$ and $x\in\RR$. 
For a function $f\in\F$, let 
$\B_j(f) := \big\{ g\in\F : \: P''_{j}(0;f,g)<0 \big\}$ and let $\C_j$ be the set of functions $f$ such that $\B_j(f)$ is nonempty.
As  $\tau_j(\cdot)$ is not convex, there exist $f,g\in\F$ such that 
$\tau_j(f+g) + \tau_j(f-g) < 2\tau_j(f)$. 
Then for $P(x):=P_{j}(x;f,g)$, we get 
\[
  0 > \tau_j(f+g) + \tau_j(f-g) - 2\tau_j(f) = 
  P(1)+P(-1)-2P(0) = \int_{-1}^1 (1-|x|) P''(x) dx .
\]
Indeed, $\int_0^1 (1-x) P''(x)dx = P(1)-P(0)-P'(0)$ and $\int_{-1}^{0}(1+x)P''(x)dx = P(-1)-P(0)+P'(0)$, which yields the last equality above.
Hence, there exists $x_0 \in [-1,1]$ such that $P_{j}''(x_0;f,g)<0$. As $P_{j}(x+x_0;f,g)=\tau_j(f+(x+x_0)g)=P_{j}(x;f+x_0g,g)$, $P_{j}''(0;f+x_0g,g)<0$. Thus, $g\in\B_j(f+x_0g)$ and $\C_j$ is nonempty.

Let $f\in \C_j$, $h\in\B_j(f)$, and $g\in\F$ and let $Q_j(x,y;f,g,h)$ be the two-variable polynomial $\tau_j(f+xg+yh)$ in real variables $x$ and $y$. 
Since $\frac{\partial^2}{\partial x^2}Q_j(x,y;f,g,h)=P_j''(x;f+yh,g)$ is negative at $x=y=0$,
there exists $\varepsilon>0$ such that
$P_j''(x;f+yh,g)$ is negative whenever $x=0$ and $y\in(-\varepsilon,\varepsilon)$. Thus, every $\C_j$ is algebraically open.

Let $\G_j$ be the class of functions $f \in \F$ such that 
$\tau_j(f)$ is strictly greater than any other $\tau_\ell(f)$, $\ell\neq j$. 
By~\ref{conv_min_2}, each $\G_j$ is nonempty. 
As $\tau_j(f+xg)-\tau_\ell(f+xg)$ is a polynomial in $x$ with a positive value at $x=0$, the set $\{f\in\F:\tau_j(f)>\tau_\ell(f)\}$ is algebraically open.
Since the intersection of finitely many algebraically open sets is again algebraically open, $\G_j$ is algebraically open, too.

Our next claim is that $\C_j$ and $\G_j$ are disjoint. 
To see this, suppose that there is $f\in\C_j\cap\G_j$. 
Then for $g\in \B_j(f)$, there exists $\varepsilon>0$ such that 
$P''_{j}(x;f,g)<0$ for $x \in (-\varepsilon,\varepsilon)$ 
and both $f + \varepsilon g$ and $f - \varepsilon g$ 
are contained in $\G_j$. 
For $P(x):=P_{j}(x;f,g)$, we get 
\begin{align*}
  \tau(f + \varepsilon g) + \tau(f - \varepsilon g) - 2\tau(f) & = 
  \tau_j(f + \varepsilon g) + \tau_j(f - \varepsilon g) - 2\tau_j(f) 
  \\ & = 
  P(\varepsilon)+P(-\varepsilon)-2P(0) = \int_{-\varepsilon}^{\varepsilon} (\varepsilon - |x|) P''(x) dx 
  < 0 ,
\end{align*}
which contradicts convexity of $\tau(\cdot)$. 

Let $\D^+:=\{f\in\F: \tau_j(f)>0 \text{ for each } j=1,2,\ldots,m\}$. 
As $\D^+$ is the intersection of finitely many algebraically open sets $\{f\in\F:\tau_j(f)>0\}$, $j=1,2,\ldots,m$,
it is algebraically open too. Moreover, it is closed under taking tensor products. That is, if $f,g\in\D^+$, then so is $f\otimes g$.

By condition \ref{conv_min_4}, $\C_j^+ := \C_j\cap \D^+$ is nonempty and algebraically open. As $\C_j\cap\G_j=\emptyset$, $\C_j^+$ must be disjoint from $\G_j$.
Let $f\in\C_j$ and let $g\in\D^+$.
Then $f \otimes g \in \C_j$, as $h\otimes g\in \B_j(f\otimes g)$ for each $h\in\B_j(f)$.
Indeed, $P_{j}(x;f\otimes g,h\otimes g) = \tau_j(f\otimes g + xh\otimes g) = \tau_j(g)P_{j}(x;f,h)$ and hence,
\begin{align*}
    P''_{j}(0;f\otimes g,h\otimes g) = \tau_j(g) \, P''_{j}(0;f,h) < 0 .
\end{align*}
In particular, if $f\in \C_1^+$ and $g\in \C_2^+$, then $f\otimes g\in \C_1\cap \C_2\cap \D^+ = \C_1^+\cap \C_2^+$. 
Iterating this with  $f_j\in\C^+_j$, $j=1,2,\ldots,m$, gives that $f_1 \otimes f_2 \otimes \ldots \otimes f_m \in\C^+$,
where $\C^+ := \C^+_1 \cap \C^+_2 \cdots \cap \C^+_m$.
Therefore,~$\C^+$ is nonempty, algebraically open, 
and disjoint from each $\G_j$, $j=1,2,\ldots,m$. 

Let $\G_{j,\ell} := \{f: \: \tau_j(f)=\tau_\ell(f)\}$ and let $\G := \cup_{j,\ell} \, \G_{j,\ell} \,$.
Then every $f\notin \cup_j\G_j$ must be in $\G$ and thus, $\C^+ \subseteq \G$. 
However, we claim that $\G$ cannot contain a nonempty algebraically open set, which contradicts the fact that $\C^+$ is a nonempty algebraically open set.

To see this, let $f$ be a function in an algebraically open subset $\oo$ of $\G$.
By~\ref{conv_min_2}, there is $h\in\F$ 
such that $\tau_j(h) < \tau_\ell(h)$. 
Then $p(x):=\tau_j(f+xh)-\tau_\ell(f+xh)$ is a polynomial of degree $k$ 
where the coefficient of $x^k$ is $\tau_j(h) - \tau_\ell(h) \neq 0$. 
As $p(x)$ cannot vanish on an entire open interval around $0$, 
the algebraically open set~$\oo$ containing $f$ must contain a function $f+xh$ such that $\tau_j(f+xh) \neq \tau_\ell(f+xh)$ for sufficiently small $x$. 
Therefore, 
$\oo$ contains a point that does not belong to $\G_{j,\ell}$.
Let $\oo':=\oo\setminus \G_{j,\ell}$. This is again algebraically open, since it is an intersection of three algebraically open sets $\{f\in\F:\tau_j(f)>\tau_\ell(f)\}$, $\{f\in\F:\tau_j(f)<\tau_\ell(f)\}$, and $\oo$.
One can repeat the same argument by replacing $\oo$ and $(j,\ell)$ by $\oo'$ and $(j',\ell')\neq (j,\ell)$, respectively. Iterating this process proves the claim.
\end{proof}

\begin{remark}\label{rem:sep}
A function $f\in\F_\RR$ is \emph{separated from zero} if there is $c>0$ such that $f \geq c$ a.e. 
Let~$\F_{\rm sep}$ be the set of all functions in $\F_\RR$ that are separated from zero. 
Then $\F_{\rm sep}$ is closed under the tensor product and algebraically open. 
By replacing $\tau_1,\ldots,\tau_m$ by weak decoration functionals and $\F=\F_\RR$ by its subset $\F_{\rm sep}$ in \cref{th:conv_min}, 
a variant of the lemma can be obtained in verbatim the same way.
This will be useful in proving~\cref{th:equivalence_asymm} for weakly norming graphs.
\end{remark}

\section{Equivalence between symmetric and asymmetric settings}\label{sec:symmetry}

In what follows, every graph is nonempty.
When considering bipartite graphs $H$, we always assume that it has no isolated vertices and also fix a bipartition $A\cup B$ and the corresponding orientation from $A$ to $B$, i.e., each edge $e\in E(H)$ is written as \emph{oriented} pair $(u,v)$, $u\in A$ and $v\in B$. 
All the homomorphisms between bipartite graphs must also respect the orientation. For example, $K_{1,2}$ and $K_{2,1}$ are considered to be nonisomorphic, unless specified otherwise as `usual' graphs. 
The graph~$mH$ denotes $m$ vertex-disjoint copies of $H$ with the corresponding bipartition. For instance, $mK_{1,2}$ has~$m$ vertices of degree two on the left-hand side and $2m$ leaves on the right hand side.

A $2$-edge-colouring $\alpha$ of $H$ always means 
a map $\alpha:E(H)\rightarrow\{0,1\}$. 
The range $\{1,2\}$ instead of $\{0,1\}$ 
will also be considered for other purposes, 
but we never call such maps $2$-edge-colourings.
If $\alpha_1,\alpha_2,\ldots,\alpha_m$ are $2$-edge-colourings of the same graph $H$, then $\alpha_1\cup\alpha_2\cup\cdots\cup\alpha_m$ denotes the colouring of $mH$ such that $i$-th copy of $H$ is coloured by $\alpha_i$. 

Let $\F_{\RR}^{\sym}$ 
be the subclass of symmetric functions in $\F_{\RR}$. 
For $f\in\F_{\CC}$, let $f^T$ denote its transpose, i.e.,  $f^T(x,y) = f(y,x)$. 
For a $2$-coloured graph $(H,\alpha)$ and a function $f\in\F_{\CC}$, 
define 
\[
  r_{H,\alpha}(f) : = \: 
  \int \prod_{e=(u,v) \in E(H)}
    f(x_u,y_v)^{\alpha(e)} \:
    f(y_v,x_u)^{1-\alpha(e)} \; d\mu^{|V(H)|} .
\]
A $2$-edge-colouring $\alpha$ can also be seen an \emph{orientation} of the (ordered) edges in $E(H)$. That is, $\alpha(u,v)=1$, then the edge is directed from $u$ to $v$, and otherwise, it is directed in the reverse way. 
For instance, let $\beta_1$, $\beta_2$, and $\beta_3$ be three orientations of the graph $K_{1,2}$, given by 
\begin{align}\label{eq:direct_K12}
    \begin{tikzpicture}[baseline=(z.base)]
    \draw 
    (0,0) node[p,label=left:$x$](x){}
    (0.7,0.17) node[p,label=right:$y$](y){}
    (0.7,-0.17) node[p,label=right:$z$](z){};
    \draw[->-=.75] (x)--(y);
    \draw[->-=.75] (x)--(z);
\end{tikzpicture}~,
\begin{tikzpicture}[baseline=(z.base)]
    \draw 
     (0,0) node[p,label=left:$x$](x){}
    (0.7,0.17) node[p,label=right:$y$](y){}
    (0.7,-0.17) node[p,label=right:$z$](z){};
    \draw[-<-=.75] (x)--(y);
    \draw[-<-=.75] (x)--(z);
\end{tikzpicture}~, \text{ and }
\begin{tikzpicture}[baseline=(z.base)]
    \draw 
    (0,0) node[p,label=left:$x$](x){}
    (0.7,0.17) node[p,label=right:$y$](y){}
    (0.7,-0.17) node[p,label=right:$z$](z){};
    \draw[-<-=.70] (x)--(y);
    \draw[->-=.75] (x)--(z);
\end{tikzpicture}~,
\end{align}
respectively.
Then the corresponding functionals are written as
\begin{align*}
    r_{H,\beta_1}(f) \! = \!\!\int f(x,y) f(x,z) \, d\mu^3,\;\; 
    r_{H,\beta_2}(f) \! = \!\!\int f(y,x) f(z,x) \, d\mu^3, 
\;\;r_{H,\beta_3}(f) \! = \!\!\int f(y,x) f(x,z) \, d\mu^3,
\end{align*}
respectively.

\medskip

For a bipartite graph $H$ and $f\in\F_{\RR}$, let 
$q_H(f) := \max_{\alpha} \left| r_{H,\alpha}(f) \right|$, 
where the maximum is taken over all $2$-edge-colourings $\alpha$ of $H$. 

A {\it decoration} of $H$ is a collection of functions $f_e\in\F_{\RR}$ 
assigned to the edges $e \in E(H)$. 
For a decoration $\{f_e\}$, 
we also define 
\[
  r_{H,\alpha}(\{f_e\}) : = \: 
  \int \prod_{e=(u,v) \in E(H)}
    f_e(x_u,y_v)^{\alpha(e)} \:
    f_e(y_v,x_u)^{1-\alpha(e)} \; d\mu^{|V(H)|} .
\]
and 
$q_H(\{f_e\}):=\max_{\alpha} |r_{H,\alpha}(\{f_e\})|$. It is straightforward to check that $r_{H,\alpha}(\{f_e\})$ is a decoration functional on $\F_\RR^k$, where $k:=\mathrm{e}(H)$.
The fact that $q_H(\{f_e\})$ is a weak decoration functional is not hard to see either, as the condition \ref{it:subadd} follows from
\begin{align*}
    q_H&(f_1,\ldots,f_{i-1},g+h,f_{i+1},\ldots,f_k)
    =\max_{\alpha}\big|r_{H,\alpha}(f_1,\ldots,f_{i-1},g+h,f_{i+1},\ldots,f_k)\big|\\
    &\leq \max_{\alpha}\Big(\big|r_{H,\alpha}(f_1,\ldots,f_{i-1},g,f_{i+1},\ldots,f_k)\big|+\big|r_{H,\alpha}(f_1,\ldots,f_{i-1},h,f_{i+1},\ldots,f_k)\big|\Big)\\
    &\leq \max_{\alpha}\big|r_{H,\alpha}(f_1,\ldots,f_{i-1},g,f_{i+1},\ldots,f_k)\big| + \max_{\alpha}\big|r_{H,\alpha}(f_1,\ldots,f_{i-1},h,f_{i+1},\ldots,f_k)\big|\\
    &=q_H(f_1,\ldots,f_{i-1},g,f_{i+1},\ldots,f_k)+q_H(f_1,\ldots,f_{i-1},h,f_{i+1},\ldots,f_k).
\end{align*}
We state these facts as a lemma for future reference.
\begin{lemma}\label{lem:deco}
For every $2$-coloured graph $(H,\alpha)$, $r_{H,\alpha}(\{f_e\})$ is a decoration functional and $q_H(\{f_e\})$ is a weak decoration functional. 
\end{lemma}
Our goal is to prove the following broader equivalence than~\cref{thm:equivalence_asymm}.
\begin{theorem}\label{th:equivalence_asymm}
Let $H$ be a connected bipartite graph. 
Then the following statements are equivalent:
\begin{enumerate}[(i)]
\item 
$H$ is (weakly) norming on $\F_{\RR}^{\sym}$;
\item 
$q_H(f)^{1/{\rm e}(H)}$ (respectively, $q_H(|f|)^{1/{\rm e}(H)}$) 
is a norm on $\F_{\RR}$;
\item 
$H$ is (weakly) norming on $\F_{\RR}$.
\end{enumerate}
\end{theorem}

The proof is verbatim the same for norming and weakly norming properties, so we will only give it for the norming property. 

\medskip

Let us sketch the proof plan for~\cref{th:equivalence_asymm}. 
The first step is to prove that if $H$ is norming on $\F_{\RR}^{\sym}$ then $q_H(f)^{1/{\rm e}(H)}$ is a norm on $\F_{\RR}$.
The next step, which relies on~\cref{th:conv_min}, is to show that $q_H(f)^{1/{\rm e}(H)}$ being a norm on $\F_{\RR}$ implies the existence of a $2$-edge-colouring $\alpha$ such that $|r_{H,\alpha}(\cdot)|^{1/{\rm e}(H)}$ is a norm.
We then conclude by showing that the $2$-edge-colouring $\alpha$ must be monochromatic.

\medskip

Firstly, we apply~\cref{th:Hatami_gen} with the choice  $\tau(\{f_e\}) := r_{H,\alpha}(\{f_e\})$, $k={\rm e}(H)$, and $\F = \F_{\RR}$.
This is possible since, for every $2$-coloured graph~$(H,\alpha)$, $r_{H,\alpha}$ is a decoration functional on~$\F_\RR^{{\rm e}(H)}$.

\begin{corollary}\label{th:Hatami_asymm}
$|r_{H,\alpha}(\cdot)|^{1/{\rm e}(H)}$ is a seminorm on $\F_{\RR}$ 
if and only if for any decoration $\{f_e\}$, 
\begin{equation*}
 |r_{H,\alpha}(\{f_e\})|^{{\rm e}(H)} \: \leq \;
 \prod_{e \in E(H)} |r_{H,\alpha}(f_e)| \: .
\end{equation*}
\end{corollary}
For a positive integer $m$, let 
\begin{align*}
    \varrho_{2m,H}(f) := \left(\sum_{\alpha} r_{H,\alpha}(f)^{2m} \right)^{\frac{1}{2m}} 
    \!\! , \;\;\;\;\;\;\;
    \varrho_{2m,H}(\{f_e\}) := \left(\sum_{\alpha} r_{H,\alpha}(\{f_e\})^{2m} \right)^{\frac{1}{2m}},
\end{align*}
where $\alpha$ ranges over all possible $2$-edge-colourings of $H$.

\begin{lemma}\label{th:Hatami_asymm_2m}
If $H$ is norming on $\F_{\RR}^{\sym}$, 
then for any decoration $\{f_e\}$
and any positive integer $m$, 
\begin{equation}\label{eq:Hatami_asymm_2m}
 \varrho_{2m,H}(\{f_e\})^{{\rm e}(H)} \: \leq \;
 \prod_{e \in E(H)} \varrho_{2m, H}(f_e) \: .
\end{equation}
\end{lemma}

\begin{proof}[{\bf Proof}]
Denote by ${\bf 1}$ the $2$-edge-colouring of $H$ 
that assigns $1$ for every edge. 
Then
\begin{align*}
 \varrho_{2m,H}(f)^{2m} = 
 \sum_{\alpha} r_{H,\alpha}(f)^{2m} = 
 \sum_{\alpha} r_{H,\alpha}(f^{\otimes 2m}) =  
 r_{H,{\bf 1}}\left(f^{\otimes 2m} + (f^T)^{\otimes 2m}\right) ,
\end{align*}
where the last equality follows from the multilinear expansion
\begin{align*}
    r_{H,{\bf 1}}\left(g + g^T\right) 
    &= \int \prod_{e=(u,v)\in E(H)}\big(g(x_u,y_v) +g(y_v,x_u)\big) d\mu^{|V(H)|}\\
    &= \int\sum_{\alpha}\prod_{e=(u,v)\in E(H)} g(x_u,y_v)^{\alpha(e)} g(y_u,x_v)^{1-\alpha(e)}d\mu^{|V(H)|}
\end{align*}
for $g=f^{\otimes 2m}$.
Similarly,
\begin{align*}
 \varrho_{2m,H}(\{f_e\})^{2m} = 
 \sum_{\alpha} r_{H,\alpha}(\{f_e\})^{2m} = 
 \sum_{\alpha} r_{H,\alpha}(\{f_e^{\otimes 2m}\}) =  
 r_{H,{\bf 1}}\left(\{f_e^{\otimes 2m} + (f_e^T)^{\otimes 2m}\}\right) .
\end{align*}
As $f_e^{\otimes 2m} + (f_e^T)^{\otimes 2m} = f_e^{\otimes 2m} +(f_e^{\otimes 2m})^T$ is symmetric and $H$ is norming on $\F_{\RR}^{\sym}$, 
\cref{eq:Hatami_asymm_2m} follows from~\cref{th:Hatami_asymm}.
\end{proof}

\begin{corollary}\label{th:Hatami_asymm_q}
If $H$ is norming on $\F_{\RR}^{\sym}$, 
then for any decoration $\{f_e\}$, 
\begin{equation*}
 q_H(\{f_e\})^{{\rm e}(H)} \: \leq \;
 \prod_{e \in E(H)} q_H(f_e) \: .
\end{equation*}
\end{corollary}

\begin{proof}[{\bf Proof}]
It immediately follows from 
$q_H(\{f_e\}) = \lim_{m\to\infty} \varrho_{2m,H}(\{f_e\})$.
\end{proof}

\begin{lemma}\label{th:convexity_asymm}
If $H$ is norming on $\F_{\RR}^{\sym}$,
then $q_H(\cdot)^{1/{\rm e}(H)}$ is a norm on $\F_{\RR}$.
\end{lemma}

\begin{proof}[{\bf Proof}]
By~\cref{lem:deco}, $q_H(\{f_e\})$ is a weak decoration functional.
Thus, \cref{th:Hatami_gen,th:Hatami_asymm_q} prove that $q_H(\cdot)^{1/{\rm e}(H)}$ is a seminorm and it remains to check that if $q_H(f)=0$ 
then $f=0$ a.e. 
Indeed, in this case $r_{H,\alpha}(f)=0$ for every $2$-edge-colouring $\alpha$, 
including the monochromatic colourings ${\bf 1}$ and ${\bf 0}$. 
Then $t_H(f \otimes f^T) = t_H(f)\, t_H(f^T) = r_{H,{\bf 1}}(f)\, r_{H,{\bf 0}}(f) = 0$. 
As $H$ is norming on $\F_{\RR}^{\sym}$, 
and $f \otimes f^T$ is symmetric, 
we have $f \otimes f^T = 0$ a.e., 
and hence, $f=0$ a.e. 
\end{proof}

As sketched, \cref{th:conv_min} will prove the existence of a $2$-edge-colouring $\alpha$ such that $r_{H,\alpha}(\cdot)^{1/{\rm e}(H)}$ is a norm by using $q_{H,\alpha}(\cdot)^{1/{\rm e}(H)}$ being a norm. To this end, we need to verify the conditions in the statement of~\cref{th:conv_min}. The next lemma will help on verifying~\ref{conv_min_4}.
\begin{lemma}\label{th:vanishing_asymm}
Let $\alpha_1,\ldots,\alpha_m$ be $2$-colourings of a graph $H$ and let $\G$ be an algebraically open subset of $\F_{\RR}$. 
Then there exists $h\in\G$ such that 
$r_{H,\alpha_i}(h) \neq 0$ for all $i=1,\ldots,m$.
\end{lemma}
\begin{proof}[{\bf Proof}]
Let $\alpha$ be the $2$-edge-colouring $\alpha_1\cup\cdots\cup\alpha_m$ of $mH$ and let $f$ be the constant function~$1$ so that $r_{mH,\alpha}(f) = 1$. 
Let $g\in\G$. Then there exists $\varepsilon > 0$ 
such that $\G$ contains all functions $g+xf$ 
with $x \in (-\varepsilon,\varepsilon)$. 
Now $P(x) := r_{mH,\alpha}(g+xf)$ 
is a polynomial of degree $m{\rm e}(H)$, 
since the coefficient $r_{mH,\alpha}(g+xf)$ 
of $x^{m{\rm e}(H)}$ is 1.
Thus, $P(x)$ has only finite number of zeros 
on the whole interval $(-\varepsilon,\varepsilon)$,
and therefore, there exist some $c\in(-\varepsilon,\varepsilon)$ 
and $h=g+cf\in\G$ such that 
$\prod_{i=1}^m r_{H,\alpha_i}(h) = r_{mH,\alpha}(h) = P(c) \neq 0$.
\end{proof}

The following lemma will be the final ingredient in proving $H$ is norming on $\F_\RR$, assuming some $r_{H,\alpha}(\cdot)^{1/{\rm e}(H)}$ is a (semi-)norm.

\begin{lemma}\label{th:monochrom_asymm}
If $H$ is connected 
and $|r_{H,\alpha}(\cdot)|^{1/{\rm e}(H)}$ 
is a seminorm on $\F_{\RR}$, 
then $\alpha$ is monochromatic.
\end{lemma}

\begin{proof}[{\bf Proof}]
Suppose to the contrary that 
$\alpha$ is not monochromatic. 
Let $m$ and $k$ be the numbers of edges of colours $1$ and $0$ in $\alpha$, respectively. 
In particular, $m,k>0$ and $m+k = {\rm e}(H)$. 
Denote by ${\bf 1}$ the $2$-edge-colouring of $H$ 
that assigns $1$ to every edge. 
Fix a function $f\in\F_\RR$. For each edge $e$, assign 
$f_e = f$ if $\alpha(e) = 1$ and $f_e = f^T$ otherwise. 
Then $t_H(f) = r_{H,{\bf 1}}(f) = r_{H,\alpha}(\{f_e\})$. 
\cref{th:Hatami_asymm} then gives
$|t_H(f)|^{m+k} \leq |r_{H,\alpha}(f)|^m \: |r_{H,\alpha}(f^T)|^k$, 
and
$|t_H(f^T)|^{m+k} \leq |r_{H,\alpha}(f^T)|^m \: |r_{H,\alpha}(f)|^k$. 
Hence,
\begin{equation}\label{eq:monochrom_1}
  |t_H(f) \: t_H(f^T)| \: \leq \:
  |r_{H,\alpha}(f) \: r_{H,\alpha}(f^T)| \, .
\end{equation}
Let $\beta_1$, $\beta_2$, and $\beta_3$ be given by~\eqref{eq:direct_K12}.
For $h\in\F_{\RR}$, set $I_i:=r_{K_{1,2},\beta_i}(h)$ for brevity.
Then
\begin{equation}\label{eq:monochrom_2}
  I_1 + I_2 - 2I_3 \: =
  \int \left( \int (h(x,y) - h(y,x)) dy \right)^2 dx 
  \: \geq \: 0 \, .
\end{equation}
As the equality above holds if and only if $\int (h(x,y) - h(y,x)) dy=0$ a.e., one may choose $h$ in such a way that 
the inequality in \cref{eq:monochrom_2} is strict.
Let $f_{\varepsilon} := 1 + \varepsilon h$. Then
\begin{align}\label{eq:t_H}
    t_{H}(f_\varepsilon) = 1&+\varepsilon {\rm e}(H)\int h(x,y)+\varepsilon^2 \left(I_1\sum_{a\in A} \binom{\deg(a)}{2}  + I_2\sum_{b\in B} \binom{\deg(b)}{2}\right)\nonumber
    \\ &+\varepsilon^2 M_{H}\left(\int h(x,y)\right)^2+O(\varepsilon^3),
\end{align}
where $M_H$ is the number of two-edge matchings in $H$. We may also write
\begin{align}\label{eq:t_H^T}
    t_{H}(f_\varepsilon^T) = 1&+\varepsilon {\rm e}(H)\int h(x,y)+\varepsilon^2 \left(I_2\sum_{a\in A} \binom{\deg(a)}{2}  + I_1\sum_{b\in B} \binom{\deg(b)}{2}\right)\nonumber
    \\ &+\varepsilon^2 M_{H}\left(\int h(x,y)\right)^2+O(\varepsilon^3).
\end{align}
For each vertex $a\in A$, let $d^+(a)$ (resp. $d^-(a)$) be the number of incident edges $e=(a,v)$ such that $\alpha(e)=1$ (resp. $\alpha(e)=0$). For $b\in B$, let $d^+(b)$ (resp. $d^-(b)$) be the number of incident edges $e=(u,b)$ such that $\alpha(e)=0$ (resp. $\alpha(e)=1$). That is, we count the in-degree and the out-degree for each vertex when considering $\alpha$ as an orientation.
Then 
\begin{align}\label{eq:r_H}
     r_{H,\alpha}(f_{\varepsilon}) = 1&+\varepsilon {\rm e}(H)\int h(x,y)
     +\varepsilon^2 \sum_{v\in V(H)}\left(\binom{d^+(v)}{2}I_1+\binom{d^-(v)}{2}I_2 +d^+(v)d^-(v)I_3\right)\nonumber\\ 
     &+\varepsilon^2 M_{H}\left(\int h(x,y)\right)^2+O(\varepsilon^3),
\end{align}
and similarly,
\begin{align}\label{eq:r_H^T}
     r_{H,\alpha}(f_{\varepsilon}^T) = 1&+\varepsilon {\rm e}(H)\int h(x,y)
     +\varepsilon^2 \sum_{v\in V(H)}\left(\binom{d^+(v)}{2}I_2+\binom{d^-(v)}{2}I_1 +d^+(v)d^-(v)I_3\right)\nonumber\\ 
     &+\varepsilon^2 M_{H}\left(\int h(x,y)\right)^2+O(\varepsilon^3),
\end{align}
As $\deg(v)=d^+(v)+d^-(v)$, 
we get $\binom{\deg(v)}{2}=\binom{d^+(v)}{2}+\binom{d^-(v)}{2}+d^+(v)d^-(v)$. Thus, subtracting~\eqref{eq:t_H} from~\eqref{eq:r_H} yields
\begin{align}\label{eq:monochrom_3}
  r_{H,\alpha}(f_{\varepsilon}) \:=\: t_H(f_{\varepsilon})
    \; &+ \;
    \varepsilon^2\sum_{a\in A}\left(\binom{d^-(a)}{2}(I_2-I_1) +d^+(a)d^-(a)(I_3-I_1)\right)\nonumber \\
    &+\;\varepsilon^2\sum_{b\in B}\left(\binom{d^+(b)}{2}(I_1-I_2)+d^+(b)d^-(b)(I_3-I_2)\right)
  +O(\varepsilon^3) \, ,
\end{align}
and similarly, subtracting~\eqref{eq:t_H^T} from~\eqref{eq:r_H^T} yields
\begin{align}\label{eq:monochrom_4}
  r_{H,\alpha}(f_{\varepsilon}^T) \:=\: t_H(f_{\varepsilon}^T)
    \; &+ \;
    \varepsilon^2\sum_{a\in A}\left(\binom{d^-(a)}{2}(I_1-I_2) +d^+(a)d^-(a)(I_3-I_2)\right)\nonumber \\
    &+\;\varepsilon^2\sum_{b\in B}\left(\binom{d^+(b)}{2}(I_2-I_1)+d^+(b)d^-(b)(I_3-I_1)\right)
  +O(\varepsilon^3) \, .
\end{align}
Let $c_1:=\sum_{a\in A}d^+(a)d^-(a)$, $c_2:=\sum_{b\in B}d^+(b)d^-(b)$, $d_1:=\sum_{a\in A}\binom{d^-(a)}{2}$, and $d_2:=\sum_{b\in B}\binom{d^+(b)}{2}$.
Multiplying \cref{eq:monochrom_3,eq:monochrom_4} then gives 
\begin{align*}
  r_{H,\alpha}(f_\varepsilon) \, r_{H,\alpha}(f_\varepsilon^T) \: &- \: t_H(f_\varepsilon) \, t_H(f_\varepsilon^T) \\
  \: =  \;
 & t_H(f_\varepsilon^T)   \, [c_1 (I_3-I_1) \,+\, c_2 (I_3-I_2) \,+\, d_1(I_2-I_1) \,+\, d_2(I_1-I_2)] \, \varepsilon^2
  \\  + \:
 & t_H(f_\varepsilon) \, [c_1 (I_3-I_2) \,+\, c_2 (I_3-I_1) \,+\, d_1(I_1-I_2) \,+\, d_2(I_2-I_1)] \, \varepsilon^2
    + \:
   O(\varepsilon^3) \, .
\end{align*}
As $t_H(f_\varepsilon) = 1 + O(\varepsilon)$ 
and $t_H(f_\varepsilon^T) = 1 + O(\varepsilon)$, we get 
\begin{equation}\label{eq:monochrom_5}
  r_{H,\alpha}(f_{\varepsilon})r_{H,\alpha}(f_{\varepsilon}^T) \:-\: t_H(f_{\varepsilon})t_H(f_{\varepsilon}^T)
    \; = \;
  (c_1+c_2) \: (2I_3 - I_1 - I_2) \: \varepsilon^2 \: + \:
  O(\varepsilon^3) \, .
\end{equation}
As $H$ is connected and $\alpha$ is not monochromatic, $c_1+c_2 > 0$. Since we chose $h$ such that $2I_3 - I_1 - I_2 < 0$ 
and both $t_H(f_\varepsilon) t_H(f_\varepsilon^T)$ 
and $r_{H,\alpha}(f_\varepsilon) r_{H,\alpha}(f_\varepsilon^T)$ 
are nonnegative, 
choosing small enough $\varepsilon>0$ makes 
\cref{eq:monochrom_5} contradict \cref{eq:monochrom_1}.
\end{proof}

\begin{proof}[{\bf Proof of \cref{th:equivalence_asymm}}]
(iii) $\Rightarrow$ (i). This is trivial by definition. 

\noindent (i)$ \Rightarrow$ (ii). This follows from \cref{th:convexity_asymm}.

\noindent (ii) $\Rightarrow$ (iii). 
By~\cref{lem:deco}, $q_H(\{f_e\})$ is a weak decoration functional. Thus, \cref{th:convexity_gen} proves that  
$q_H(\cdot)$ is convex. 
Let $\A=\{\alpha_1,\ldots,\alpha_m\}$ 
be a minimal collection of $2$-edge-colourings of $H$ such that
$\max_j \left| r_{H,\alpha_j}(f) \right| = q_H(f)$ 
for any $f\in\F_{\RR}$. 
Set $\tau_j(f) := r_{H,\alpha_j}(f)^2$ 
and $\tau(f) := q_H(f)^2$. 
Then each~$\tau_j(\cdot)$ is nonnegative and $\tau(\cdot)$ is convex on $\F_{\RR}$, as $q_H(\cdot)$ is convex. Moreover, 
$\tau(f) = \max_j \tau_j(f)$.
Thus, the condition~\ref{conv_min_1} of \cref{th:conv_min} is satisfied.
By minimality of $\A$, 
the condition~\ref{conv_min_2} is satisfied too.
As $\prod_{j=1}^m \tau_j(f) = 
r_{mH,\alpha_1\cup\alpha_2\cup\cdots\cup\alpha_m}(f)^2$, 
it follows from \cref{th:vanishing_asymm} that 
condition~\ref{conv_min_4} of \cref{th:conv_min} is satisfied. 

By \cref{th:conv_min}, there is $j\in\{1,2,\ldots,m\}$ such that 
$\tau_j(f) = r_{2H,\alpha_j\cup\alpha_j}(f)$ is convex. 
Then by \cref{th:convexity_gen}, 
$r_{2H,\alpha_j\cup\alpha_j}(\cdot)^{\frac{1}{2{\rm e}(H)}}$ 
is a seminorm on $\F_{\RR}$. 
Since $|r_{H,\alpha_j}(f)|^{1/{\rm e}(H)} =
r_{2H,\alpha_j\cup\alpha_j}(f)^{\frac{1}{2{\rm e}(H)}}$, 
$|r_{H,\alpha_j}(f)|^{1/{\rm e}(H)}$ is a seminorm too. 
As $H$ is connected, 
$\alpha_j$ is monochromatic by \cref{th:monochrom_asymm}. 
We may assume that $\alpha_j = {\bf 1}$, 
so $|t_H(f)|^{1/{\rm e}(H)} = |r_{H,{\bf 1}}(f)|^{1/{\rm e}(H)}$ 
is a seminorm on $\F_{\RR}$. 

It remains to prove $t_H(f)$ is nonzero unless $f=0$ a.e.
Suppose $t_H(f)=0$. Then $r_{H,{\bf 1}}(f\otimes f^T)=t_H(f)t_H(f^T)=0$.
On the other hand, by~\cref{th:Hatami_asymm}, $|r_{H,{\bf 1}}(f\otimes f^T)|=q_H(f\otimes f^T)$, 
since $f\otimes f^T$ is symmetric, which makes all $r_{H,\alpha}(f\otimes f^T)$ have the same value.
Thus, $q_H(f\otimes f^T)=0$. This means $f\otimes f^T=0$ a.e., since $q_H(\cdot)^{1/{\rm e}(H)}$ defines a norm and therefore, $f=0$ a.e. 
\end{proof}

\begin{remark}
The proof of weakly norming case in  \cref{th:equivalence_asymm} follows from the corresponding  analogues of \cref{th:Hatami_asymm,th:Hatami_asymm_2m,th:Hatami_asymm_q,th:convexity_asymm,th:vanishing_asymm,th:monochrom_asymm}.
The key adjustment here is to assess the convexity and the weakly norming property on the subset of nonnegative functions rather than on the whole $\F_{\RR}$. 
Then it is not hard to follow the original arguments to obtain the variants of these lemmas. 
In the proof of \cref{th:equivalence_asymm} itself, 
one should also use a modified version of \cref{th:conv_min}, which was discussed in \cref{rem:sep}. 
\end{remark}

We conclude this section by 
giving a much simpler proof of Hatami's bi-regularity theorem as an application of \cref{th:equivalence_asymm}. 
\begin{theorem}[{\cite[Theorem~2.10(ii)]{H10}}]\label{thm:biregular}
Let $H$ be weakly norming on $\F_\RR^\sym$.
If $H$ is connected, then it is bi-regular, i.e., every vertex on the same side of the bipartition has the same degree.
\end{theorem}
\begin{proof}
Fix a bipartition $A\cup B$ of $H$. By~\cref{th:equivalence_asymm}, $H$ is weakly norming on $\F_{\RR}$.
Let $f$ be the indicator function of the box $[0,1/2)\times[0,1]$. Then $t_H(f) = 2^{-|A|}$. For $a\in A$, let $S_a$ be the star induced on $a\cup N_H(a)$ and let $\{f_e\}$ be the decoration such that $f_e=f$ if $e\in E(S_a)$ and $f_e=1$ otherwise. Then
\begin{align*}
    t_H(\{f_e\})=t_{S_a}(f)=1/2 \leq t_H(f)^{\deg(a)/{\rm e}(H)} = 2^{-|A|\deg(a)/{\rm e}(H)}.
\end{align*}
Indeed, the inequality follows from Hatami's inequality~\cite[Theorem~2.8]{H10}, which generalises to~\cref{eq:Hatami_gen}.
Thus, $|A|\deg(a)/{\rm e}(H)\leq 1$. If $\deg(a)$ is larger than the average degree $d={\rm e}(H)/|A|$ on $A$, then ${\rm e}(H)=d|A|<\deg(a)|A|\leq {\rm e}(H)$, which gives a contradiction. Thus, $\deg(a)=d$ for each $a\in A$. The same conclusion follows for each $b\in B$ if we replace $f$ by $f^T$.
\end{proof}

We remark that Hatami's Theorem~2.10(ii) in~\cite{H10} is stated for weakly norming graphs on $\F_\RR$, 
and is weaker than our \cref{thm:biregular}. 
However, Hatami's proof uses only symmetric functions and hence obtains essentially the same result as ours, 
albeit in a more complicated way.

\section{Properties of complex-norming graphs}\label{sec:characterisation}
In this section, we obtain various properties of complex-norming graphs, including~\cref{thm:unique}. Many of the results will provide tools in what follows, while already interesting in their own right.

\medskip

It is sometimes necessary to consider the case when $\|.\|_H$ is a \emph{seminorm}. If so, we say that the corresponding $H$ is real- or complex-seminorming, depending on the vector space we consider. Fortunately, the complete classification of real-seminorming but not real-norming graphs is known. We extend this to the complex-valued case.

\begin{theorem}\label{th:seminorm}
Let H be a graph with no isolated vertices. Then the following are equivalent:
\begin{enumerate}[(i)]
    \item\label{it:c-semi} $H$ is complex-seminorming but not complex-norming;
    \item\label{it:r-semi}  $H$ is real-seminorming but not real-norming;
    \item\label{it:star} $H$ is isomorphic to either $mK_{1,1}$, $mK_{1,2d}$, or $mK_{2d,1}$.
\end{enumerate}
\end{theorem}

\begin{proof}[\bf{Proof}]
It is already known~\cite{GHL19,L12} that $mK_{1,1}$, $mK_{1,2d}$, and $mK_{2d,1}$ are the only real-seminorming graphs that are not real-norming.\footnote{In~\cite{L12}, only connected graphs were considered for this statement; however, \cite[Theorem 1.2]{GHL19} proves that each component must be isomorphic to one another unless it is an isolated vertex.}
Thus, it remains to prove that \ref{it:c-semi} is equivalent to the others.
To prove \ref{it:star}$\Rightarrow$ \ref{it:c-semi}, one may easily see that $mK_{1,1}$, $mK_{1,2d}$, and $mK_{2d,1}$ are not complex-norming, as they are not real-norming. 
To prove that they are complex-seminorming, we first exclude the trivial case $mK_{1,1}$ that defines the same seminorm as $|\int f|$. We further assume for the remaining cases that $H$ is connected, i.e., $H$ is isomorphic to either $K_{1,2d}$ or $K_{2d,1}$.
A $2$-edge-colouring $\alpha$ that colours exactly half the edges by $0$ induces the $L^{2d}$-norm on the space of one variable functions $g(y)=\int f(x,y) dx$ or $h(x)=\int f(x,y) dy$, depending on the orientation of $H$. This defines a seminorm on $\F_\CC$.

Next, we prove \ref{it:c-semi}$\Rightarrow$\ref{it:r-semi}.
If $H$ is complex-seminorming but not complex-norming, there exists nonzero $f\in\F_\CC$ such that $\|f\|_{H,\alpha}=0$. Then $\|\overline{f}\|_{H,\alpha}=0$ too. 
As $(H,\alpha)$ is complex-seminorming, $\|f+\overline{f}\|_{H,\alpha} \leq \
\|f\|_{H,\alpha} + \|\overline{f}\|_{H,\alpha} = 0$. Since $f+\overline{f}$ is a real-valued function with $t_H(f+\overline{f})=0$, $f$~is nonzero but $f+\overline{f}=0$ a.e. unless $H$ is not real-norming.
But then, as $f$ is purely imaginary, $|t_H(if)|=\|f\|_{H,\alpha}^{{\rm e}(H)}=0$, which proves that $H$ is not real norming.
 On the other hand, $H$ is real-seminorming since it is complex-seminorming.
\end{proof}

Theorem~\ref{th:seminorm} tells us that, with the complete list of few exceptions, the statements are intended to explain the (complex-)norming property, even though the (complex-)seminorming property will often be used.

\medskip

Since a complex-norming graph $H$ is always real-norming, every complex-norming $H$ inherits all the properties of real-norming graphs. A (not necessarily connected) graph $H$ is called 
\emph{Eulerian} if every vertex in $H$ has even degree.
\begin{lemma}[{\cite[Observation~2.5(ii)]{L12}}]\label{lem:eulerian}
Every complex-norming graph $H$ is Eulerian. In particular, as $H$ is bipartite, it has an even number of edges. 
\end{lemma}
Furthermore, every complex-norming $H$ is edge-transitive, since every weakly norming graph is edge-transitive by \cite[Theorem~1]{Sidorenko20}. We further extend this result, by introducing various symmetry notions for $2$-edge-colourings $\alpha$.

Suppose that a graph $H$ is edge-transitive and Eulerian, and let $\alpha$ be its $2$-edge-colouring.
We then say that $\alpha$ is {\it balanced} if
every vertex is incident to the equal number of edges of each colour. 
A balanced $\alpha$ is {\it self-conjugate} if
there is an automorphism of $H$ which reverses the colours of edges. 
If~$\alpha$ is self-conjugate, then 
$t_{H,\alpha}(f)=t_{H,\alpha}(\overline{f})=\overline{t_{H,\alpha}(f)}$, 
and hence, the value of $t_{H,\alpha}(f)$ is always real.
We say that a balanced $\alpha$ is {\it transitive} if,
for any two edges $a,b$ of the same (resp. opposite) colour, 
there exists an automorphism of $H$ 
which preserves (resp. reverses) the colours of edges and maps $a$ to $b$. 

A transitive $2$-edge-colouring is in particular self-conjugate and, by definition, 
a self-conjugate $2$-edge-colouring is balanced.
We remark that it is essential to assume balancedness in the definitions of both self-conjugacy and transitivity.
\cref{fig:C8} shows an example $2$-edge-colouring of $C_8$ that satisfies all the conditions to be transitive but balancedness.
\begin{figure}
    \centering
    \begin{tikzpicture}[thick,scale=1.5]
        \node[circle, draw, fill=black, scale=.3] (x) at (0,0){};
        \node[circle, draw, fill=black, scale=.3] (xy) at (-0.2,0.5){};
        \node[circle, draw, fill=black, scale=.3] (y) at (0,1){};
        \node[circle, draw, fill=black, scale=.3] (yz) at (0.5,1.2){};
        \node[circle, draw, fill=black, scale=.3] (z) at (1,1){};
        \node[circle, draw, fill=black, scale=.3] (zw) at (1.2,0.5){};
        \node[circle, draw, fill=black, scale=.3] (w) at (1,0){};
        \node[circle, draw, fill=black, scale=.3] (wx) at (0.5,-0.2){};
        
        \begin{scope}[thick,dashed,,opacity=0.6]
            \draw (x) -- (xy) -- (y);
            \draw (z) -- (zw) -- (w);
        \end{scope}
        \draw (y) -- (yz) -- (z);
        \draw (x) -- (wx) -- (w);
    \end{tikzpicture}
    \caption{A non-balanced but `symmetric' $2$-edge-colouring of $C_8$}
    \label{fig:C8}
\end{figure}
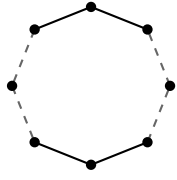

\medskip

The following result extends edge transitivity of weakly norming graphs and provides an important tool to prove~\cref{thm:equiv} as well as~\cref{thm:cubes} and its variants.

\begin{theorem}\label{th:edge-transitive}
If $(H,\alpha)$ is complex-norming, then $\alpha$ is transitive. 
\end{theorem}

In order to prove \cref{th:edge-transitive}, we need 
a few auxiliary facts about real- and complex-norming graphs. Two $2$-coloured graphs $(H,\alpha)$ and $(G,\beta)$ are \emph{isomorphic} if there is a graph isomorphism $\phi$ between $H$ and $G$ that preserves the colourings, i.e., $\alpha(e)=\beta(\phi(e))$ for $e\in E(H)$. 

\begin{lemma}[{\cite[Lemma~2.16]{H09}}]\label{th:isomorphism}
If $t_{H,\alpha}(f) = t_{G,\beta}(f)$ for all $f\in\F_{\CC}$, 
then $(H,\alpha)$ and $(G,\beta)$ are isomorphic. 
\end{lemma}

\begin{lemma}[{\cite[Exercise~14.8]{L12}}]\label{lem:positive}
Every real-norming graph $H$ satisfies $t_H(f)\geq 0$ 
for any $f\in\F_\RR$.
\end{lemma}
In fact, the original statement in~\cite{L12} is slightly stronger in the sense that it uses real-seminorming graphs $H$. The proof of \cref{lem:positive} is simpler than \cite[Exercise~14.8]{L12} and almost identical to the corresponding part of the proof of \cref{lem:real+}. We hence omit the proof.

\begin{lemma}\label{th:both_colours}
Let $f(x,y) := \exp(2\pi(x+y)i)$. 
Then $t_{H,\alpha}(f)=1$ if $\alpha$ is balanced, and $t_{H,\alpha}(f)=0$ otherwise.
\end{lemma}

\begin{proof}[\bf{Proof}]
For a vertex $v$ of $H$ and $j=0,1$, let $d_j(v)$ denote the the number of edges of colour $j$ incident to $v$. 
Then 
\begin{equation}\label{eq:balanced}
  \prod_{(u,v) \in E(H)} f(x_u,x_v) \: = 
  \prod_{v \in V(H)} \exp((d_1(v) \! - \! d_0(v))\, 2\pi i \, x_v) \, .
\end{equation}
If $\alpha$ is balanced, then the right hand side of \cref{eq:balanced} is equal to $1$. 
Otherwise, there is $v$ such that $d_1(v)-d_0(v) \neq 0$ 
and $\int_0^1 \exp((d_1(v)-d_0(v)) 2\pi i x_v) dx_v = 0$. 
\end{proof}

\begin{lemma}\label{th:balanced}
If $(H,\alpha)$ is a complex-norming $2$-coloured graph, then $\alpha$ must be balanced. 
\end{lemma}

\begin{proof}[\bf{Proof}]
Let $f(x,y) := \exp(2\pi(x+y)i)$. 
Suppose that $\alpha$ is not balanced. 
By \cref{th:both_colours}, 
$t_{H,\alpha}(f) = 0$ and $t_{H,\alpha}(\overline{f}) = t_{H,\alpha}(-f) = 0$. 
Let $g(x,y) := \frac{1}{2}(f+\overline{f}) = \cos(2\pi(x+y))$. 
Since $g$ is a nonzero real-valued function and~$H$ is real-norming, $t_{H,\alpha}(g)$ must be positive by~\cref{lem:positive}.
However, it contradicts the triangle inequality
$t_{H,\alpha}(2g) \leq \left(
t_{H,\alpha}(f)^{1/{\rm e}(H)} + 
t_{H,\alpha}(\overline{f})^{1/{\rm e}(H)}
\right)^{{\rm e}(H)}$. 
\end{proof}
In particular, if $(H,\alpha)$ is complex-norming, 
it must contain exactly ${\rm e}(H)/2$ edges of each colour.

\begin{corollary}\label{cor:positive}
Let $(H,\alpha)$ be a complex-norming $2$-coloured graph.
Then for every $c\in\CC$ and $f\in\F_\CC$, $t_{H,\alpha}(cf)=|c|^{{\rm e}(H)}t_{H,\alpha}(f)$. In particular, $t_{H,\alpha}(c)>0$ whenever $c$ is a nonzero complex number.
\end{corollary}

Recall that a decoration of $H$ is a collection of functions $f_e\in\F_{\CC}$ 
assigned to the edges $e \in E(H)$. 
Let
\[
 t_{H,\alpha}(\{f_e\}) : = \;
  \int \prod_{e=(u_i,v_j) \in E(H)}
    f_e(x_i,y_j)^{\alpha(e)} \:
    \overline{f_e(x_i,y_j)}^{1-\alpha(e)} \: d\mu^{|V(H)|} .
\]
\cref{cor:positive} now proves the condition~\ref{it:scalar} in the definition of decoration functional. It is then easy to see that $t_{H,\alpha}(\{f_e\})$ is a decoration functional.
Applying~\cref{th:Hatami_gen,th:convexity_gen} with the choice $\tau(\{f_e\}) := t_{H,\alpha}(\{f_e\})$, $k={\rm e}(H)$, and $\F = \F_{\CC}$
gives the following consequences:
\begin{corollary}
\label{th:Hatami}
A $2$-coloured graph $(H,\alpha)$ is complex-seminorming if and only if the following holds: for any decoration $\{f_e\}$, 
\begin{equation}\label{eq:Hatami}
 \left| t_{H,\alpha}(\{f_e\}) \right|^{{\rm e}(H)} \: \leq \;
 \prod_{e \in E(H)} |t_{H,\alpha}(f_e)| \: .
\end{equation}
\end{corollary}

An easy fact about complex numbers will be useful.
\begin{lemma}\label{th:L_conjugate}
Let $x,y\in\CC$ and $x \neq \overline{y}$. 
There exists $z\in\CC$ such that ${\rm Re}(zx -\overline{z}y) > 0$. 
\end{lemma}

\begin{proof}[\bf{Proof}]
As $x \neq \overline{y}$, then either ${\rm Re}(x) \neq {\rm Re}(y)$, 
or ${\rm Im}(x+y) \neq 0$. 
If ${\rm Re}(x) > {\rm Re}(y)$ select $z=1$. 
If ${\rm Re}(x) < {\rm Re}(y)$ select $z=-1$. 
If ${\rm Im}(x+y) > 0$, select $z=-i$. 
If ${\rm Im}(x+y) < 0$, select $z=i$. 
\end{proof}

The \emph{conjugate} colouring $\overbar{\alpha}$ of a $2$-edge-colouring $\alpha$ is defined by $\overbar{\alpha} := 1-\alpha$, i.e., each colour assigned to an edge is flipped.
Two $2$-coloured graphs $(H,\alpha)$ and $(H,\overbar{\alpha})$ are said to be \emph{conjugates} to each other.
If $H'$ is a subgraph of $H$, the $2$-coloured graph $(H',\alpha)$ means that the edges of $H'$ is coloured by $\alpha$ restricted to the edge set of $H'$.
We are now ready to prove the key lemma in proving Theorem~\ref{th:edge-transitive}.

\begin{lemma}\label{th:derivative}
Let $(H,\alpha)$ be a complex-norming $2$-coloured graph. 
If edges $e_1,e_2$ are of the same colour, 
then  
$t_{H\setminus e_1,\alpha}(f) = t_{H\setminus e_2,\alpha}(f)$ 
for any $f\in\F_{\CC}$. 
Otherwise if $e_1,e_2$ are of the opposite colours, 
then  
$t_{H\setminus e_1,\alpha}(f) \: \overline{t_{H,\alpha}(f)} = 
\overline{t_{H\setminus e_2,\alpha}(f)} \: t_{H,\alpha}(f)$ 
for any $f\in\F_{\CC}$. 
\end{lemma}

\begin{proof}[\bf{Proof}]
We may assume $t_{H,\alpha}(f)\neq 0$, as $t_{H,\alpha}(f)=0$ if and only if $f=0$ a.e., which already satisfies the conclusion.
Write $E(H) = \{e_1,e_2,\ldots,e_k\}$. 
Let $z$ be a complex number and let $\varepsilon \geq 0$ be a small real number. Both numbers will be chosen later.
Denote $\tau_j = t_{H\setminus e_j,\alpha}(f)$ for brevity. 
Then 
\begin{align*}
  t_{H,\alpha}(f + \varepsilon z) \: = \:
  t_{H,\alpha}(f) \: + \:
  \varepsilon \sum_{j=1}^k \left(
    \alpha(e_j) z \: + \: \overline{\alpha}(e_j) \overline{z}
  \right) \: \tau_j
  \: + \: O(\varepsilon^2)
\end{align*}
and hence,
\begin{align}\label{eq:monochr}
  t_{H,\alpha}(f + \varepsilon z) \; 
  t_{H,\alpha}(f - \varepsilon z) \: = \:
  t_{H,\alpha}(f)^2 \: + \:
  O(\varepsilon^2) \: .
\end{align}
Define a decoration $\{f_{e_j}\}$ as 
$f_{e_{_1}} = f + \varepsilon z$, 
$f_{e_{_2}} = f - \varepsilon z$, 
$f_{e_j} = f$ for $j \geq 3$. 
Let
\begin{equation*}
  \Delta \: := \: 
  \left(
    \alpha(e_1) z \: + \: \overline{\alpha}(e_1) \overline{z}
  \right) \: \tau_1
    \: - \:
  \left(
    \alpha(e_2) z \: + \: \overline{\alpha}(e_2) \overline{z}
  \right) \: \tau_2
    \: .
\end{equation*}
Then
\begin{align}\label{eq:deco}
  t_{H,\alpha}(\{f_{e_j}\}) \: = \:
  t_{H,\alpha}(f) \: + \:
  \varepsilon \Delta \: + \:
  O(\varepsilon^2) \: .
\end{align}
Substituting~\eqref{eq:monochr} and~\eqref{eq:deco} into~\cref{eq:Hatami} yields
\begin{equation*}
  \left| 
  t_{H,\alpha}(f) \: + \:
  \varepsilon \Delta \: + \:
  O(\varepsilon^2) 
  \right|^{k}
    \: \leq \:
  \left|
  t_{H,\alpha}(f)^2 \: + \:
  O(\varepsilon^2)
  \right|
  \: |t_{H,\alpha}(f)|^{k-2}
  \: ,
\end{equation*}
and hence,
\begin{equation*}
  \left| 
  t_{H,\alpha}(f)^{k} \: + \:
  k \varepsilon \Delta \: t_{H,\alpha}(f)^{k-1} \: + \:
  O(\varepsilon^2) 
  \right|
    \: \leq \:
  |t_{H,\alpha}(f)|^{k} \: + \:
  O(\varepsilon^2) 
  \: .
\end{equation*}
As $t_{H,\alpha}(f) \neq 0$, dividing both sides by $|t_{H,\alpha}(f)|^{k}$ gives
\begin{equation}\label{eq:derivative}
  \left\vert
  1 \: + \:
  \frac{k \varepsilon \Delta}{t_{H,\alpha}(f)}
  \: + \:
  O(\varepsilon^2)  \right\vert
    \: \leq \:
  1 \: + \: O(\varepsilon^2) 
  \: .
\end{equation}

Suppose that $e_1$ and $e_2$ are of the same colour, 
say $\alpha(e_1) = \alpha(e_2) = 1$. 
Then $\Delta = z (\tau_1 - \tau_2)$.
If $\tau_1 \neq \tau_2$, then  
select $z = t_{H,\alpha}(f) (\overline{\tau_1} - \overline{\tau_2})$
so that $\Delta / t_{H,\alpha}(f)=|\tau_1-\tau_2|^2$. 
Then for small enough  $\varepsilon>0$, 
the left hand side in \cref{eq:derivative} 
is greater than the right hand side. 
Therefore, $\tau_1 = \tau_2$. 

In the case when $\alpha(e_1)=1$ and $\alpha(e_2)=0$, 
$\Delta = z \tau_1 - \overline{z} \tau_2$. 
If $\tau_1 / t_{H,\alpha}(f) \neq \overline{\tau_2 / t_{H,\alpha}(f)}$, 
choose~$z$ by setting $x=\tau_1 / t_{H,\alpha}(f)$ and $y=\tau_2 / t_{H,\alpha}(f)$ in \cref{th:L_conjugate} so that ${\rm Re}(zx-\overline{z}y)={\rm Re}(\Delta/t_{H,\alpha}(f))>0$.
Then again, for $\varepsilon$ small enough, 
the left hand side in \cref{eq:derivative} 
is greater than the right hand side. 
Therefore, 
$\tau_1 / t_{H,\alpha}(f) = \overline{\tau_2 / t_{H,\alpha}(f)}$. 
\end{proof}

\begin{proof}[\bf{Proof of \cref{th:edge-transitive}}]
By~\cref{th:balanced}, $\alpha$ is balanced.
Let~$e_1$ and $e_2$ be edges of the same colour. 
It follows from \cref{th:derivative,th:isomorphism} that
there is an isomorphism $\phi$ from $(H\setminus e_1,\alpha)$ to $(H\setminus e_2,\alpha)$. 
Then $\phi$ must map each edge of $H\setminus e_1$ 
to an edge of $H\setminus e_2$ of the same colour.
As $H$ is Eulerian by \cref{lem:eulerian}, 
the ends of $e_1$ and $e_2$ are the only vertices of odd degrees 
in $H\setminus e_1$ and $H\setminus e_2$, respectively.
Then $\phi$ must map $e_1$ to~$e_2$, which proves that $\phi$ is a colour-preserving automorphism of $H$. 

Let~$H_1$ and $H_2$ be connected components of $H$. We claim that $(H_1,\alpha)$ and $(H_2,\alpha)$ are isomorphic.
Since $\alpha$ is balanced, every connected component of $H$ must contain edges of both colours. Let $e_1\in E(H_1)$ and $e_2\in E(H_2)$ be edges of the same colour.
Then by what we have just shown, there is a colour-preserving isomorphism $\phi$ between $(H\setminus e_1,\alpha)$ and $(H\setminus e_2,\alpha)$ that maps $e_1$ to $e_2$. Then this $\phi$ maps $H_1$ to $H_2$ and thus, $(H_1,\alpha)$ and $(H_2,\alpha)$ are isomorphic.
In particular, we may assume that $H$ is connected.

\medskip

Let $e_1$ and $e_2$ be edges of the opposite colours and let $(H',\beta)$ be the union of vertex-disjoint copies of 
$(H\setminus e_1,\alpha)$ and $(H,\overline{\alpha})$, and let $(H'',\gamma)$ be the union of vertex-disjoint copies of 
$(H\setminus e_2,\overline{\alpha})$ and $(H,\alpha)$. 
By \cref{th:derivative}, 
$t_{H',\beta}(f) = t_{H'',\gamma}(f)$ for every~$f\in\F_\CC$. 
Hence, by \cref{th:isomorphism}, 
$(H',\beta)$ and $(H'',\gamma)$ are isomorphic.
Since $H$ is Eulerian, every edge is contained in a cycle. In particular, $H\setminus e_1$ and $H\setminus e_2$ are connected, and thus, they are unique components with ${\rm e}(H)-1$ edges in $H'$ and $H''$, respectively.
Therefore, $(H\setminus e_1,\alpha)$ must be isomorphic to $(H\setminus e_2,\overline{\alpha})$. Let $\phi$ be an isomorphism between the two. 
Analogously to the same-colour case, $\phi$ maps $e_1$ to $e_2$ and thus, it extends to a colour-reversing automorphism of $H$ that maps $e_1$ to $e_2$.
\end{proof}

An application of \cref{th:edge-transitive} extends~\cref{lem:positive} to complex-valued functions.
\begin{theorem}\label{lem:real+}
If $(H,\alpha)$ is a complex-norming $2$-coloured graph, 
then $t_{H,\alpha}(f)$ is a positive real number, 
unless $f$ is equal to $0$ a.e.
\end{theorem}

\begin{proof}[\bf{Proof}]
By \cref{th:edge-transitive}, $\alpha$ is transitive. 
Hence, $t_{H,\alpha}(f) = t_{H,\alpha}(\overline{f}) = 
\overline{t_{H,\alpha}(f)}$, which means $t_{H,\alpha}(f)$ is always real.

Next, we will prove that if $t_{H,\alpha}(f) = 0$, then 
$f$ is equal to $0$ a.e. 
Indeed, consider $g = f + \overline{f}$. 
By the triangle inequality, 
$\|g\|_{{H,\alpha}} \leq 
\|f\|_{H,\alpha} + \|\overline{f}\|_{H,\alpha} = 0$.
As $g$ is real-valued, and $H$ is real-norming, 
$\|g\|_{H,\alpha}=\|g\|_H=0$ and hence, $g=0$ a.e. 
Since $g=2\cdot {\rm Re}(f)$, $h = i \cdot f$ must be real-valued a.e.
By \cref{cor:positive}, $t_H(h)=t_{H,\alpha}(h) = t_{H,\alpha}(f) = 0$. 
Therefore, $h=0$ a.e., and thus, $f=0$ a.e. 

Finally, we are going to prove that $t_{H,\alpha}(f) \geq 0$. 
Suppose to the contrary that there is a function $f$ such that $t_{H,\alpha}(f) < 0$.
Let $f_\xi:=\xi f +(1-\xi)$. As $\xi$ decreases from $1$ to $0$, $t_{H,\alpha}(f_\xi)$ continuously moves from $t_{H,\alpha}(f)$ to 1, so there exists a value $\xi_0\in (0,1)$ such that $t_{H,\alpha}(f_{\xi_0})=0$.
By the previous paragraph, $f_{\xi_0}=0$ a.e., and thus, $f$ must be equal to a constant $c=(\xi_0-1)/\xi_0$ a.e. However, by~\cref{cor:positive}, $t_{H,\alpha}(c)$ for $c\neq 0$ must be positive.
\end{proof}

We have just proved that our less restrictive definition of complex-norming is in fact equivalent to the one given by Hatami \cite{H09}.

\begin{corollary}\label{th:no_abs_value}
A $2$-coloured graph $(H,\alpha)$ is complex-norming if and only if 
$t_{H,\alpha}(\cdot)^{1/{\rm e}(H)}$ defines a norm on~$\F_{\CC}$.
\end{corollary}

As a corollary, we also obtain a complex-seminorming version of \cite[Theorem~4.1]{LS19}.

\begin{corollary}\label{th:convexity}
Let $(H,\alpha)$ be a $2$-coloured graph. 
Then $(H,\alpha)$ is complex-seminorming if and only if 
$t_{H,\alpha}(\cdot)$ is convex on $\F_\CC$.
\end{corollary}

The next result shows that $t_{H,\alpha}(\cdot)$ 
attains the maximum modulus amongst $t_{H,\beta}(\cdot)$ 
for all possible $2$-edge-colourings $\beta$ of $H$. 

\begin{theorem}\label{th:P1}
If a coloured graph $(H,\alpha)$ is norming, 
then for any $2$-edge-colouring $\beta$ 
and any $f\in\F_{\CC}$,
\[
  \left|t_{H,\beta}(f)\right| \: \leq \: t_{H,\alpha}(f) \: .
\]
\end{theorem}

\begin{proof}[\bf{Proof}]
Fix a function $f$, and for each edge $e$ assign 
$f_e = f$ if $\alpha(e) = \beta(e)$, 
and $f_e = \overline{f}$ otherwise. 
Then
\[
 t_{H,\beta}(f) \: = \: t_{H,\alpha}(\{f_e\}) .
\]
As $t_{H,\alpha}(f)$ is always real by~\cref{lem:real+}, 
\cref{eq:Hatami} implies $|t_{H,\beta}(f)| \leq t_{H,\alpha}(f)$. 
\end{proof}

\cref{thm:unique} now follows as a corollary.

\begin{proof}[\bf{Proof of~\cref{thm:unique}}]
Let $\alpha$ and $\beta$ be $2$-edge-colourings of $H$ such that both $(H,\alpha)$ and $(H,\beta)$ are complex-norming.
It follows from \cref{lem:real+,th:P1} that 
$t_{H,\alpha}(f) = t_{H,\beta}(f)$ for any $f\in\F_\CC$. 
Then by \cref{th:isomorphism}, 
$(H,\alpha)$ and $(H,\beta)$ must be isomorphic.
\end{proof}

\section{Equivalence between real- and complex-norming}\label{sec:blind}
As mentioned in the introduction, our proof of \cref{thm:equiv} does not construct any norming $2$-edge-colouring~$\alpha$ for $H$.
Instead, as done in~\cref{sec:symmetry}, we define a functional $s_H(\cdot)$ on $\F_\CC$ that does \emph{not} depend on the particular choice of a $2$-edge-colouring but `bridges' the two properties, real- and complex-norming.

\medskip

For a bipartite graph $H$ and $f\in\F_{\CC}$, let $s_H(f) := \max_{\alpha} \left| t_{H,\alpha}(f) \right|$, 
where the maximum is taken over all $2$-edge-colourings $\alpha$ on $H$. 
This definition is partly motivated by \cref{th:P1}, which states that, if $(H,\alpha)$ is complex-norming, then $s_H(f)$ must be the same as $t_{H,\alpha}(f)$.
We also define $s_H(\{f_e\}):=\max_{\alpha} |t_{H,\alpha}(\{f_e\})|$ for decorations $\{f_e\}$.
We omit the proof of the following properties of $s_H(\{f_e\})$ and $t_{H,\alpha}(\{f_e\})$ that are easy to check and moreover, analogous to~\cref{lem:deco}.
\begin{lemma}\label{lem:deco_complex}
For every $2$-coloured graph $(H,\alpha)$, $t_{H,\alpha}(\{f_e\})$ is a decoration functional and $s_H(\{f_e\})$ is a weak decoration functional. 
\end{lemma}

Our goal is to prove a broader equivalence statement than~\cref{thm:equiv}, which resembles the strategy taken in~\cref{sec:symmetry} that proved~\cref{thm:equivalence_asymm} by using \cref{th:equivalence_asymm}. 

\begin{theorem}\label{th:equivalence}
Let $H$ be a bipartite graph. Then the following statements are equivalent:
\begin{enumerate}[(i)]
\item 
$H$ is real-norming;
\item 
$s_H(\cdot)^{1/{\rm e}(H)}$ is a norm on $\F_{\CC}$;
\item 
$H$ is complex-norming.
\end{enumerate}
\end{theorem}

The proof strategy for~\cref{th:equivalence} also resembles that for~\cref{th:equivalence_asymm}.
To sketch roughly, we shall first prove that if~$H$ is real-norming then $s_H(f)^{1/{\rm e}(H)}$ is a norm on $\F_\CC$. 
Next, analogously to the proof of~\cref{th:equivalence_asymm}, \cref{th:conv_min} will tell us that there exists a $2$-edge-colouring $\alpha$ such that $t_{H,\alpha}(\cdot)^{1/{\rm e}(H)}$ is a norm.
In addition, we also conclude by~\cref{thm:unique,th:edge-transitive} that $\alpha$ is transitive and unique up to isomorphism.
This additional observation leads us to obtain the key lemma to prove some graphs are not norming in the subsequent sections.
\begin{corollary}\label{th:real_transitive}
If a graph does not admit transitive $2$-edge-colourings, 
then it is not real-norming.
\end{corollary}
 
 We now start proving $s_H(f)^{1/{\rm e}(H)}$ is a norm on $\F_\CC$ under the assumption that $H$ is real-norming.
\begin{lemma}\label{th:s_H}
If $H$ is real-norming, 
then for any complex-valued decoration $\{f_e\}$,
\begin{equation*}
s_H(\{f_e\})^{{\rm e}(H)}\leq \prod_{e\in E(H)} s_H(f_e).
\end{equation*}
\end{lemma}

\begin{proof}[{\bf Proof}]
Let $g_e := f_e\otimes\overline{f_e}$. As  $t_{H,\alpha}(\{g_e\}) = t_{H,\alpha}(\{f_e\}) \cdot t_{H,\alpha}(\{\overline{f_e}\}) =|t_{H,\alpha}(\{f_e\})|^2$, $t_{H,\alpha}(\{g_e\})$ 
is real and nonnegative for any colouring $\alpha$. 
Denote by ${\bf 1}$ the $2$-edge-colouring of $H$ that assigns $1$ for every edge. 
For any positive integer $m$, 
\begin{align*}
  \sum_{\alpha} t_{H,\alpha}(g)^m & = 
  \sum_{\alpha} t_{H,\alpha}(g^{\otimes m}) = 
  t_{H,{\bf 1}}(g^{\otimes m}+\overline{g^{\otimes m}}) =
  t_H\big(2\,{\rm Re}(g^{\otimes m})\big)~\text{ and }
\\
  \sum_{\alpha} t_{H,\alpha}(\{g_e\})^m & = 
  \sum_{\alpha} t_{H,\alpha}(\{g_e^{\otimes m}\}) =
  t_{H,{\bf 1}}(\{g_e^{\otimes m}\}+\overline{\{g_e^{\otimes m}\}}) \: = \:
  t_H\big(2\,{\rm Re}(\{g_e^{\otimes m}\})\big).
\end{align*}
By \cref{th:Hatami_gen},
\[
  t_H\big(2~{\rm Re}(\{g_e\}^{\otimes m})\big)^{{\rm e}(H)} \leq
  \prod_{e \in E(H)} t_H\big(2~{\rm Re}(g_e^{\otimes m})\big), 
\]
and hence,
\[
  \left(\sum_{\alpha} t_{H,\alpha}(\{g_e\})^m\right)^{{\rm e}(H)} 
  \leq
  \prod_{e \in E(H)} \left(\sum_{\alpha} t_{H,\alpha}(g_e)^m\right) .
\]
Then the statement of the lemma follows from 
$s_H(f_e)^2 = s_H(g_e) = \lim_{m\to\infty}\left(\sum_{\alpha} t_{H,\alpha}(g_e)^m\right)^{1/m}$ 
and 
$s_H(\{f_e\})^2 = s_H(\{g_e\}) = \lim_{m\to\infty}\left(\sum_{\alpha} t_{H,\alpha}(\{g_e\})^m\right)^{1/m}$. 
\end{proof}

\begin{lemma}\label{cor:triangle_s_H}
Let $H$ be a real-norming graph. Then $s_H(\cdot)^{1/{\rm e}(H)}$ defines a norm on $\F_\CC$.
\end{lemma}

\begin{proof}[{\bf Proof}]
By \cref{lem:deco_complex}, $s_H(\{f_e\})$ defines a weak decoration functional.
Therefore, by \cref{th:Hatami_gen,th:s_H}, $s_H(\cdot)^{1/{\rm e}(H)}$ is a seminorm and it remains to check that if $s_H(f)=0$ 
then $f=0$ a.e. 
Indeed, as $s_H(\overline{f}) = s_H(f)$, 
we get by the triangle inequality, 
$t_H(2{\rm Re}(f)) = s_H(2{\rm Re}(f)) = s_H(f + \overline{f}) \leq 
\left(s_H(f)^{1/{\rm e}(H)} + 
s_H(\overline{f})^{1/{\rm e}(H)}\right)^{{\rm e}(H)}$ = 0. 
As $H$ is norming, ${\rm Re}(f)=0$ a.e. 
Since 
$s_H(-i \cdot f) = s_H(f) = 0$, 
we conclude that ${\rm Im}(f)$ is also equal to zero a.e. 
\end{proof}

The next lemma, parallel to~\cref{th:vanishing_asymm}, is to verify \ref{conv_min_4} in~\cref{th:conv_min}.

\begin{lemma}\label{th:vanishing}
Let $\alpha_1,\ldots,\alpha_m$ be $2$-colourings of a graph $H$ 
with ${\rm e}(H) > 0$. 
Let $\G$ be an algebraically open set in $\F_{\CC}$. 
Then there exists $h\in\G$ such that 
$t_{H,\alpha_i}(h) \neq 0$ for all $i=1,\ldots,m$.
\end{lemma}

\begin{proof}[\bf{Proof}]
Let $\alpha$ be the $2$-edge-colouring $\alpha_1\cup\cdots\cup\alpha_m$ of $mH$ for brevity and let $f$ be the constant function $1$ so that $t_{mH,\alpha}(f) = 1$ and let $g\in\G$. 
Then there exists $\varepsilon > 0$ 
such that $\G$ contains all functions $g+xf$ 
with $x \in (-\varepsilon,\varepsilon)$. 

Now $P(x) := t_{mH,\alpha}(g+xf)$ 
is a polynomial of degree $m{\rm e}(H)$ with possibly complex coefficients, 
since the coefficient $t_{mH,\alpha}(g+xf)$ 
of $x^{m{\rm e}(H)}$ is 
$t_{mH,\alpha}(f) \neq 0$. 
Thus,~$P(x)$ has only finite number of zeros 
on the whole interval $(-\varepsilon,\varepsilon)$,
and therefore, there exists some $c\in(-\varepsilon,\varepsilon)$ 
and $h=g+cf\in\G$ such that 
$\prod_{i=1}^m t_{H,\alpha_i}(h) = t_{mH,\alpha}(h) = P(c) \neq 0$.
\end{proof}

\begin{proof}[{\bf Proof of \cref{th:equivalence}}]
(iii)$ \Rightarrow$ (i). This is trivial by definition. 

\noindent (i) $\Rightarrow$ (ii). This follows from \cref{cor:triangle_s_H}.

\noindent (ii) $\Rightarrow$ (iii). 
By~\cref{lem:deco_complex}, $s_H(\{f_e\})$ is a weak decoration functional. As $s_H(\cdot)^{1/{\rm e}(H)}$ is a norm, $s_H(\cdot)$ is convex by~\cref{th:convexity_gen}. 
Let $\A=\{\alpha_1,\ldots,\alpha_m\}$ 
be a minimal (by size) collection of $2$-edge-colourings of $H$ such that
$\max_j \left| t_{H,\alpha_j}(f) \right| = s_H(f)$ 
for any $f\in\F_{\CC}$. 
Set $\tau_j(f) := \left| t_{H,\alpha_j}(f) \right|^2$ 
and $\tau(f) := s_H(f)^2$. 
Then $\tau(\cdot)$ is convex and 
$\tau(f) = \max_j \tau_j(f)$, which satisfies the condition~\ref{conv_min_1} of \cref{th:conv_min}.
By minimality of $\A$, the condition~\ref{conv_min_2} is satisfied too.
As $\prod_{j=1}^m \tau_j(f) = 
|t_{mH,\alpha_1\cup\alpha_2\cup\cdots\cup\alpha_j}(f)|^2$,~\cref{th:vanishing} verifies~\ref{conv_min_4}. 

By \cref{th:conv_min}, there is $j\in\{1,2,\ldots,m\}$ such that 
$\tau_j(f) = t_{2H,\alpha_j\cup\overline{\alpha_j}}(f)$ is convex. 
Since
$\tau(\{f_e\}) := t_{2H,\alpha_j\cup\overline{\alpha_j}}(\{f_e\})$ 
is a decoration functional, $t_{2H,\alpha_j\cup\overline{\alpha_j}}(\cdot)^{\frac{1}{2{\rm e}(H)}}$ 
is a seminorm on $\F_{\CC}$ by~\cref{th:convexity_gen}. 
Recall that for any $f\in\F_\RR$, $s_H(f)=|t_H(f)|$ always holds. As $s_H(\cdot)^{1/{\rm e}(H)}$ is a norm on $\F_\RR$, $H$ is real-norming and hence, $2H$ is not complex-seminorming by~\cref{th:seminorm}.
Thus, $t_{2H,\alpha_j\cup\overline{\alpha_j}}(\cdot)^{1/(2{\rm e}(H))}$ 
is a norm on $\F_\CC$. 
As $|t_{H,\alpha_j}(f)|^{1/{\rm e}(H)} =
t_{2H,\alpha_j\cup\overline{\alpha_j}}(f)^{\frac{1}{2{\rm e}(H)}}$, 
$(H,\alpha_j)$ is complex-norming.
\end{proof}

\section{Colour patterns of the cycles in norming graphs}\label{sec:cycles}

From this section, we again call a $2$-coloured graph $(H,\alpha)$ or a graph $H$ \emph{norming} if it is complex-norming, as done before in~\cref{sec:characterisation,sec:blind}.
To prove that a coloured graph $(H,\alpha)$ 
is not norming, we repeatedly use \cref{th:P1} 
with $f = 1 + \varepsilon h$, 
where $\varepsilon$ is a real number with small absolute value and $h$ is a suitably chosen function. 
The functions we shall find exceptionally useful in this regard are: 
\begin{equation*}
  h_0(x,y) := \exp(2\pi(x+y)i) ~~\text{ and } \;\;
  h_k(x,y) := 2 \exp\left(\frac{2\pi i}{k}\right) 
  \, \cos(2\pi(x+y)) 
  \;\;\;\mbox{with}\; k \geq 1 .
\end{equation*}
The standard multilinear expansion gives
\begin{equation}\label{eq:expansion}
  t_{H,\alpha}(1+\varepsilon h) = 
  \sum_{F \subseteq H} \, \varepsilon^{{\rm e}(F)} 
  t_{F,\alpha}(h) ,
\end{equation}
where the sum is taken over all subgraphs $F$ of $H$. 

For a (not necessarily balanced) $2$-edge-colouring $\alpha$, we say that a subgraph $F$ of $H$ is \emph{balanced} with respect to $\alpha$ if each vertex is incident to equal numbers of edges of $F$ of each colour. 
For example, by~\cref{th:balanced}, $H$ is balanced with respect to $\alpha$ if $(H,\alpha)$ is norming.
If $h=h_0$, then by \cref{th:both_colours}, $t_{F,\alpha}(h_0)=1$ if $F$ is balanced with respect to $\alpha$, 
or $t_{F,\alpha}(h_0)=0$ otherwise. 
Let $g=g(H)$ be the girth of $H$. 
Since $H$ is bipartite, all cycles are of even length. 
As every balanced subgraph is Eulerian, colour-alternating cycles of length $g$ and $g+2$ (if exist) are the smallest and the second smallest balanced subgraphs, respectively.
Let $\kappa_\ell(\alpha)$ be the number of colour-alternating cycles of length $\ell$ in $(H,\alpha)$.
Then
\begin{equation}\label{eq:cycles}
  t_{H,\alpha}(1+\varepsilon h_0) \: = \:
  1 \, 
  + \; \varepsilon^{g} \kappa_g(\alpha) \, 
  + \; \varepsilon^{g+2} \kappa_{g+2}(\alpha) \, 
  + \; O(\varepsilon^{g+4}) .
\end{equation}
\Cref{th:P1,eq:cycles} then yield 

\begin{corollary}\label{th:h1}
Let $(H,\alpha)$ be norming and let $g$ be the girth of $H$.
Then $\alpha$ has at least as many colour-alternating cycles 
of length $g$ as any other $2$-colourings of $H$.
\end{corollary}

By~\cite[Example~7.2]{L12}, $t_F\big(2\cos(2\pi(x-y))\big)$ evaluates to the number of \emph{Eulerian orientations} of $F$.
In our language, this is equivalent to the number of balanced $2$-edge-colourings $\alpha$, if we consider $\alpha$ as an orientation of the edges as done in~\cref{sec:symmetry}.
The same result also holds for the function $\cos(2\pi(x+y))$ instead of $\cos(2\pi(x-y))$ by changing the variable $y$ to $1-y$.
Thus, $t_{F,\alpha}(h_k)\neq 0$ if and only if $F$ is Eulerian and moreover,
it is explicitly computable when $F=C_{2\ell}$.
For a $2$-edge-colouring $\alpha$ of $C_{2k}$, let $|\alpha|$ be the number of edges of colour $1$ in $\alpha$. 
Then, as there are exactly two balanced $2$-edge-colourings of $C_{2\ell}$ that are conjugates to each other,
\begin{equation}\label{eq:h_k}
  t_{C_{2\ell},\alpha}(h_{k}) \: = \: 2\cdot t_{C_{2\ell},\alpha}\left(\exp\left(\frac{2\pi i}{k}\right)\right)\: =\:
  2 \exp\left(\frac{2\pi i}{k}\right)^{|\alpha| - (2\ell-|\alpha|)} = \:
  2 \exp\left(\frac{4\pi i (|\alpha|-\ell)}{k}\right) .
\end{equation}
Thus, by letting $\ell=g/2$, we obtain an analogous formula to \eqref{eq:cycles} for $k \geq 1$: 
\begin{align}\label{eq:cycles_general}
   t_{H,\alpha}(1+\varepsilon h_{k}) \: &= \:
   1+ \varepsilon^g \sum_{F\subseteq H,F\cong C_g} t_{F,\alpha}(h_k) +O(\varepsilon^{g+2})\nonumber\\
   &=1+ 2\varepsilon^g \sum_{F\subseteq H,F\cong C_g} \exp\left(\frac{4\pi i(|\alpha_F|-g/2)}{k}\right) +O(\varepsilon^{g+2}),
\end{align}
where $\alpha_F$ denotes the restriction of $\alpha$ onto the subgraph $F$.

\begin{theorem}\label{th:h3}
Let $(H,\alpha)$ be norming and let $g$ be the girth of $H$. Then every cycle of length $g$ is either monochromatic or contains $g/2$ edges of each colour.
\end{theorem}

\begin{proof}[\bf{Proof}]
For a $2$-colouring $\beta$ on $H$, denote by $s_m(\beta)$ the number of those $g$-cycles that have exactly~$m$ edges of colour $1$ in $\beta$. Then~\eqref{eq:cycles_general} with $k=g$ gives
\begin{align*}
    t_{H,\alpha}(1+\varepsilon h_{g}) \: &= \:
  1 \, + \; 
   2 \varepsilon^{g} \sum_{m=0}^g s_m(\alpha) \,
    \exp\left(\frac{4\pi i m}{g}\right) \, +O(\varepsilon^{g+2}).
\end{align*}
Suppose to the contrary that there is $m\notin\{0,\, g/2,\, g\}$ 
such that $s_m(\alpha) > 0$. Let $s$ be the number of $g$-cycles in $H$. 
Then by
${\rm Re}\left(\exp\left(\frac{4\pi i m}{g}\right)\right) < 1$,
\[
  {\rm Re}\left(\sum_{m=0}^g s_m(\alpha) \,
    \exp\left(\frac{4\pi i m}{g}\right)\right)
  \: < \: \sum_{m=0}^g s_m(\alpha) 
  \: = \: s \, .
\]
Since $t_{H,\alpha}(\cdot)$ is always real by~\cref{lem:real+}, $t_{H,\alpha}(1+\varepsilon h_{g}) = 
{\rm Re}\left(t_{H,\alpha}(1+\varepsilon h_{g})\right)=1+2c\varepsilon^g +O(\varepsilon^{g+2})$ for some $c\in\RR$ strictly smaller than $s$.

Let $\mathbf{1}$ be the monochromatic colouring of $H$ with the constant value $1$. 
Then $s_m(\mathbf{1}) = 0$ for $m< g$ and $s_g(\mathbf{1})=s$. 
As $\exp\left(\frac{4\pi i m}{g}\right) = 1$ 
for $m=g$,
\[
  \sum_{m=0}^g s_m(\mathbf{1}) \,
    \exp\left(\frac{4\pi i m}{g}\right)
  \: = \: \sum_{m=0}^g s_m(\mathbf{1})
  \: = \: s \, .
\]
Therefore, $|t_{H,\mathbf{1}}(1+\varepsilon h_{g})|=1+2s\varepsilon^g +O(\varepsilon^{g+2})$, which means
$t_{H,\alpha}(1+\varepsilon h_{g}) < 
|t_{H,\mathbf{1}}(1+\varepsilon h_{g})|$ 
for small enough~$\varepsilon>0$. This contradicts \cref{th:P1}. 
\end{proof}

\begin{corollary}\label{cor:manyC4}
Let $(H,\alpha)$ be a norming $2$-coloured graph. Let $v_1,v_2\in V(H)$ such that there is $u\in V(H)$ with $\alpha(u,v_1)=\alpha(u,v_2)$.
Then all the $2$-edge paths from $v_1$ to $v_2$ are monochromatic. 
\end{corollary}
\begin{proof}[{\bf Proof}]
Suppose to the contrary that there is a common neighbour $w$ of $v_1$ and $v_2$ such that $\alpha(w,v_1)\neq \alpha(w,v_2)$. 
Then $w\neq u$ and moreover, the $4$-cycle induced on $\{u,v_1,w,v_2\}$ has odd number of edges in each colour. This contradicts~\cref{th:h3}.
\end{proof}

Let $\alpha$ be a $2$-edge-colouring of $H$. 
Suppose that the girth~$g$ of $H$ is $4$. Let $k=8$ in~\eqref{eq:cycles_general}. 
Then each subgraph $F\cong C_g$ with $\alpha_F=\beta$ contributes the term $\exp(\pi i(|\beta|-2)/2)=-\exp(\pi i|\beta|/2)$. On the other hand, if $\alpha_F=\overline{\beta}$, then the contribution is
\begin{align*}
    \exp\left(\frac{\pi i(|\overline{\beta}|-2)}{2}\right)=-\exp\left(\frac{\pi i(2-|\beta|)}{2}\right)=-\exp\left(\frac{-\pi i|\beta|}{2}\right),
\end{align*}
whose real part $-\cos(\pi|\beta|/2)$ is the same as $-\exp(\pi i|\beta|/2)$. More precisely, $-\cos(\pi|\beta|/2)$ evaluates to $-1,0,1,0,$ or $-1$ if $|\beta|=0,1,2,3,$ or $4$, respectively.

\medskip

This motivates us to classify four different types of $2$-edge-colourings of 4-cycles, as shown in~\cref{fig:C4}, due to the conjugate pair description.
Namely, let $c_1(\alpha)$ be the number of colour-alternating $4$-cycles in $(H,\alpha)$, $c_2(\alpha)$ be the number of monochromatic $4$-cycles, 
$c_3(\alpha)$ denote the number of $4$-cycles that have a pair of adjacent edges of each colour,
and $c_4(\alpha)$ denote the number of $4$-cycles that have three edges of one colour and one edge of the other colour. 
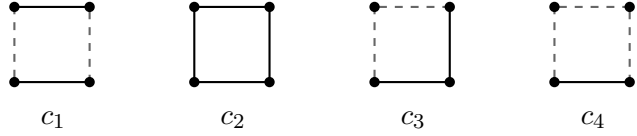
\begin{figure}
    \centering
    \begin{tikzpicture}[thick,scale=1]
        \node[circle, draw, fill=black, scale=.3] (x) at (0,0){};
        \node[circle, draw, fill=black, scale=.3] (y) at (0,1){};
        \node[circle, draw, fill=black, scale=.3] (z) at (1,1){};
        \node[circle, draw, fill=black, scale=.3] (w) at (1,0){};
        \begin{scope}[thick,dashed,,opacity=0.6]
            \draw (x) -- (y);
            \draw (z) -- (w);
        \end{scope}
        \draw (y) -- (z);
        \draw (x) -- (w);
        \node[text width=.5cm] at (0.6,-0.5) 
    {$c_1$};
    \end{tikzpicture}
    \hspace{10mm}
    \begin{tikzpicture}[thick,scale=1]
        \node[circle, draw, fill=black, scale=.3] (x) at (0,0){};
        \node[circle, draw, fill=black, scale=.3] (y) at (0,1){};
        \node[circle, draw, fill=black, scale=.3] (z) at (1,1){};
        \node[circle, draw, fill=black, scale=.3] (w) at (1,0){};
        \draw (x) -- (y) -- (z) -- (w) -- (x);
        \node[text width=.5cm] at (0.6,-0.5) 
    {$c_2$};
    \end{tikzpicture}
    \hspace{10mm}
     \begin{tikzpicture}[thick,scale=1]
        \node[circle, draw, fill=black, scale=.3] (x) at (0,0){};
        \node[circle, draw, fill=black, scale=.3] (y) at (0,1){};
        \node[circle, draw, fill=black, scale=.3] (z) at (1,1){};
        \node[circle, draw, fill=black, scale=.3] (w) at (1,0){};
        \begin{scope}[thick,dashed,,opacity=0.6]
            \draw (x) -- (y) --(z);
        \end{scope}
        \draw (z) -- (w) -- (x);
        \node[text width=.5cm] at (0.6,-0.5) 
    {$c_3$};
    \end{tikzpicture}
    \hspace{10mm}
     \begin{tikzpicture}[thick,scale=1]
        \node[circle, draw, fill=black, scale=.3] (x) at (0,0){};
        \node[circle, draw, fill=black, scale=.3] (y) at (0,1){};
        \node[circle, draw, fill=black, scale=.3] (z) at (1,1){};
        \node[circle, draw, fill=black, scale=.3] (w) at (1,0){};
        \begin{scope}[thick,dashed,,opacity=0.6]
            \draw (x) -- (y) -- (z) --(w);
        \end{scope}
        \draw (w) -- (x);
        \node[text width=.5cm] at (0.6,-0.5) 
    {$c_4$};
    \end{tikzpicture}
    \caption{Types of $2$-edge-coloured 4-cycles.}
    \label{fig:C4}
\end{figure}
Then~\eqref{eq:cycles_general} gives
\begin{align}\label{eq:h2}
  {\rm Re}\left(t_{H,\alpha}(1+\varepsilon h_8)\right) \: = \:
  1 \, 
  + \; 2
  (c_1(\alpha) + 
   c_3(\alpha) - 
   c_2(\alpha))
   \varepsilon^{4}
  + \; O(\varepsilon^{6}) ,
\end{align}
If $(H,\alpha)$ is norming, $t_{H,\alpha}(f)$ is real by~\cref{lem:real+}.
Thus, if there exists $\beta$ with $c_1(\beta)+c_3(\beta)-c_2(\beta)$ larger than $c_1(\alpha)+c_3(\alpha)-c_2(\alpha)$, \eqref{eq:h2} yields
\begin{align*}
    t_{H,\alpha}(1+\varepsilon h_8)=\mathrm{Re}(t_{H,\alpha}(1+\varepsilon h_8)) 
    < \mathrm{Re}(t_{H,\beta}(1+\varepsilon h_8))\leq |t_{H,\beta}(1+\varepsilon h_8)|,
\end{align*}
which contradicts~\cref{th:P1}.
\begin{corollary}\label{th:h2}
If a $2$-coloured graph $(H,\alpha)$ is norming, 
then $\alpha$ maximises 
$
  c_1(\alpha) + 
  c_3(\alpha) - 
  c_2(\alpha) 
$
amongst all $2$-edge-colourings of $H$ and moreover, $c_4(\alpha)=0$.
\end{corollary}

\subsection{Even-dimensional hypercubes}\label{sec:cubes}

The \emph{$d$-dimensional hypercube} $Q_d$ is the graph on $\{0,1\}^d$ 
where two vertices are adjacent if they differ in exactly one coordinate. 
In particular, $Q_2$ is isomorphic to $C_4$, which is a norming graph.

Since \cref{lem:eulerian} guarantees that every norming graph is Eulerian,
the odd-dimensional hypercubes $Q_{2d+1}$ are already not norming.
Thus, to prove~\cref{thm:cubes}, it remains to check what happens for $Q_{2d}$ with $d>1$.

\begin{theorem}\label{th:Q_2d}
For $d>1$, $Q_{2d}$ is not norming.
\end{theorem}

Our first step towards proving \cref{th:Q_2d} is to show that, for every $d>1$, there are two balanced $2$-edge-colourings of $Q_{2d}$ with different numbers of colour-alternating $4$-cycles. 

Let $\alpha_{2d}$ be the balanced $2$-edge-colouring of $Q_{2d}$ such that the colour of an edge depends only on its direction. 
That is, the edges parallel to one of the first $d$ axes receive colour $1$
and the remaining edges parallel to the last $d$ axes receive colour $0$. For $d>1$, $\alpha_{2d}$ always has monochromatic 4-cycles.

The construction of the second balanced $2$-edge-colouring $\beta_{2d}$ is a bit more sophisticated. First, we identify vertices of $Q_{2d}$ with  vectors from $\FF_2^{2d}$. 
Let $\mathbf{e}_1,\mathbf{e}_2,\ldots,\mathbf{e}_{2d}$ be the standard ordered basis of~$\FF_2^{2d}$. 
In particular, every edge in $Q_{2d}$ incident to $\mathbf{x}\in \FF_2^{2d}$ can be written as $\{\mathbf{x},\mathbf{x}+\mathbf{e}_j\}$ for some~$j$.
For $\mathbf{x}=(x_1,\ldots,x_{2d})$, let $|\mathbf{x}|=\sum_{\ell=1}^{2d}x_{\ell}$ for brevity.
Now define
\[
  \beta_{2d}(\mathbf{x}, \mathbf{x}+\mathbf{e}_j) \: := \:
  \begin{cases}
   |\mathbf{x}| + x_j + x_{j+d}, \;\mbox{if}\; j \leq d;
   \\
   |\mathbf{x}| + x_j + x_{j-d} + 1, \;\mbox{if}\; j>d.
  \end{cases}
\]
This definition is consistent as 
$\beta_{2d}(\mathbf{x}, \mathbf{x}+\mathbf{e}_j)
=\beta_{2d}(\mathbf{x}+\mathbf{e}_j, \mathbf{x})$. 
Every 4-cycle in $Q_{2d}$ is induced on the vertices $\mathbf{x},\mathbf{x}+\mathbf{e}_j,\mathbf{x}+\mathbf{e}_k,$ and $\mathbf{x}+\mathbf{e}_j+\mathbf{e}_k$ for $j<k$.
If $j<k\leq d$, then the four edges in the 4-cycle receive the colours $|\mathbf{x}| + x_j + x_{j+d}$, $|\mathbf{x}| + x_k + x_{k+d}$, $|\mathbf{x}+\mathbf{e}_j| + x_k + x_{k+d}$, and $|\mathbf{x}+\mathbf{e}_k| + x_j + x_{j+d}$, respectively. As $|\mathbf{x}+\mathbf{e}_j|=|\mathbf{x}|+1=|\mathbf{x}+\mathbf{e}_k|$, exactly two of the four edges receive the same colour.
Analogous arguments for the cases $d<j<k$ and $j\leq d<k$ yield the same conclusion that every 4-cycle has two edges in each colour.

\medskip

When $d=1$, $\alpha_2$ and $\beta_2$ are isomorphic to the alternating colouring of $C_4$. 
In fact, it is not hard to check that both $\alpha_{2d}$ and $\beta_{2d}$ are transitive for any $d$,
although our proof of \cref{th:Q_2d} does not rely 
on transitivity of either of them. 
We simply use the fact that they are balanced and, for $d>1$, 
$(Q_{2d},\alpha_{2d})$ has monochromatic $4$-cycles 
while $(Q_{2d},\beta_{2d})$ does not, i.e.,  
$c_2(\alpha_{2d}) > 0$ but $c_2(\beta_{2d})=0$. 

\begin{proof}[\bf{Proof of \cref{th:Q_2d}}]
By \cref{th:equivalence}, it is sufficient to show that 
$(Q_{2d},\gamma)$ is not norming for any $2$-edge-colouring $\gamma$. 
Suppose to the contrary that $(Q_{2d},\gamma)$ is norming.
As every $2$-edge path in $Q_{2d}$ extends to a unique 4-cycle, there are $\binom{2d}{2}2^{2d}/4=d(2d-1)2^{2d-2}$ 4-cycles in $Q_{2d}$. By \cref{th:h3}, $c_4(\gamma)=0$. 
Thus, 
\begin{align}\label{eq:c123}
    c_1(\gamma) + c_2(\gamma) + c_3(\gamma) = d(2d-1)2^{2d-2}.
\end{align}
By~\cref{th:balanced}, $\gamma$ must be balanced, so each vertex in $(Q_{2d},\gamma)$ 
is the middle vertex of $d^2$ colour-alternating $2$-edge paths.
The total number of such paths is hence $d^22^{2d}$. 

On the other hand, a colour-alternating $4$-cycle contains four colour-alternating $2$-edge paths
while a $4$-cycle with two adjacent edges of each colour 
contains two such paths.
Hence, 
\begin{align*}
    4c_1(\gamma) + 2c_3(\gamma) = d^2 2^{2d}.
\end{align*}
Substituting this into~\eqref{eq:c123}, we get
$c_1(\gamma) + c_2(\gamma) + (d^2 2^{2d-1} - 2c_1(\gamma)) = d(2d-1)2^{2d-2}$. Therefore,
\begin{align*}
    c_1(\gamma) = c_2(\gamma) + d2^{2d-2} ~~\text{ and }~~
    c_1(\gamma) + c_3(\gamma) - c_2(\gamma) 
= d(2d-1)2^{2d-2} - 2c_2(\gamma).
\end{align*}
The same equations also hold for $\alpha_{2d}$ or $\beta_{2d}$, as both are balanced and $c_4(\alpha_{2d})=c_4(\beta_{2d})=0$.

If $c_2(\gamma)=0$, then $c_1(\gamma) < c_1(\alpha_{2d})$. However, \cref{th:h1} requires~$\gamma$ to maximise~$c_1$, which is a contradiction.
Otherwise if $c_2(\gamma)>0$, then $c_1(\gamma) + c_3(\gamma) - c_2(\gamma) < 
c_1(\beta_{2d}) + c_3(\beta_{2d}) - c_2(\beta_{2d})$.
This contradicts \cref{th:h2}, which requires $\gamma$ to maximise $c_1+c_3-c_2$.
\end{proof}

\subsection{The 1-subdivision of complete graphs}\label{sec:1sub_K_n}

The \emph{1-subdivision} of a graph $G$ is the graph obtained by replacing every edge of $G$ by a path of length two.
Let $H_n$ be the $1$-subdivision of $K_n$. 
As $H_n$ is isomorphic to the incidence graph of the 
$1$- and $0$-dimensional faces of the $(n-1)$-dimensional simplex, it is weakly norming~\cite[Theorem~1.1]{CL16}.

For $n=2,3$, $H_2=K_{1,2}$ is seminorming and $H_3=C_6$ is norming.
The main result of this subsection is that $H_3$ is the only example that is norming amongst all the 1-subdivision of complete graphs.

\begin{theorem}\label{th:subdivision}
For $n>3$, the $1$-subdivision of the complete graph $K_n$ is not norming.
\end{theorem}
As a norming graph must be Eulerian, it is enough to verify~\cref{th:subdivision} for the case $n=2d+1$.
Suppose that $H$ is the $1$-subdivision of (not necessarily bipartite) $2d$-regular graph $G$ and $\alpha$ is a balanced $2$-edge-colouring of $H$. 
If $uv$ is an edge of $G$ and $w$ is its subdividing vertex in $H$, 
then $\alpha(u,w)$ and $\alpha(v,w)$ must be distinct. 
Let the edges of $G$ be oriented in such a way that an edge $uv\in E(G)$ is ordered as the arc $(u,v)$ if $\alpha(u,w)=1$ and $\alpha(v,w)=0$, 
or as $(v,u)$ otherwise. 
Then the resulting digraph $\vec{G}(\alpha)$ is {\it regular} 
in the sense that the out-degrees and the in-degrees of all its vertices are equal to $d$.
In particular, if $G=K_{2d+1}$, then $\vec{G}(\alpha)$ is a regular tournament on $2d+1$ vertices.

By \cref{th:edge-transitive}, if $(H,\alpha)$ is norming, 
then $\alpha$ must be transitive. 
It is straightforward to see that $\alpha$ is transitive 
if and only if the digraph $\vec{G}(\alpha)$ is arc-transitive. 
The graphs $G$ that admits an arc-transitive orientation are called 
{\it semi-transitive} and have been studied extensively (see, for example,~\cite{W04}).
For another example, Berggren \cite{B72} proved that 
an arc-transitive $n$-vertex tournament exists 
if and only if $n \equiv 3 \pmod{4}$ is a power of a prime and is isomorphic to the one with the vertex-set $GF(n)$ 
where $(x,y)$ is an arc if $y-x = z^2$ for some $z \neq 0$. 
This already reduces~\cref{th:subdivision} for~$H_n$ to the special case that $2d+1$ is a power of a prime and $d$ is odd.

\medskip

However, we do not rely on Berggren's result and give a self-contained proof by computing the number of directed $3$-cycles and $4$-cycles, i.e., 
\begin{tikzpicture}[baseline=(z.base)]
    \draw 
    (0,0) node[p](x){}
    (0.25,0.27) node[p](y){}
    (0.5,0) node[p](z){};
    \draw[-<-=.70] (x)--(y);
    \draw[->-=.70] (x)--(z);
    \draw[-<-=.5] (y)--(z);
\end{tikzpicture}\: and 
\begin{tikzpicture}[baseline=0pt]
   \draw 
     (0,0.12) node[p](0){}
    (0.3,0.28) node[p](1){}
    (0.6,0.12) node[p](2){}
    (0.3,-0.04) node[p](3){};
    \draw[->-=.65] (0)--(1);
    \draw[->-=.5] (1)--(2);
    \draw[->-=.65] (2)--(3);
    \draw[->-=.5] (3)--(0);
\end{tikzpicture}\:,
 in arc-transitive or regular tournaments. 
For a tournament~$T$, let $\vec{\kappa}_3(T)$ and $\vec{\kappa}_4(T)$ be 
the number of (unlabelled) directed $3$- and $4$-cycles in $T$, respectively.

\begin{lemma}\label{lem:direct_K3}
Let $n=2d+1$. For every $n$-vertex regular tournament $T$, 
$\vec{\kappa}_3(T)=\frac{1}{6} nd(d+1)$.
\end{lemma}
\begin{proof}[{\bf Proof}]
Since each vertex has out-degree $d$ and in-degree $d$, there are $nd^2$ two-arc directed paths, i.e., \begin{tikzpicture}[baseline=(z.base)]
    \draw 
    (0,0.07) node[p](x){}
    (-0.3,-0.15) node[p](y){}
    (0.3,-0.15) node[p](z){};
    \draw[-<-=.70] (x)--(y);
    \draw[->-=.75] (x)--(z);
\end{tikzpicture}\:, in $T$. When counting these paths, each copy of a directed $3$-cycle contributes three copies, whereas each of the other triples contributes exactly one. Hence, $nd^2= 3\vec{\kappa}_3(T) + \left(\binom{n}{3} - \vec{\kappa}_3(T) \right)$, which gives $\vec{\kappa}_3(T) =  \frac{1}{6} nd(d+1)$.
\end{proof}

\begin{lemma}\label{th:transitive_tournament}
Let $n=2d+1$. For every $n$-vertex arc-transitive tournament $T$,
$\vec{\kappa}_4(T) = \frac{3}{4} n \binom{d+1}{3}$.
\end{lemma}

\begin{proof}[{\bf Proof}]
Fix an arc $(x,y)$ from $x$ to $y$.
By~\cref{lem:direct_K3} and arc-transitivity, $(x,y)$ extends to $\frac{d+1}{2}$ directed $3$-cycles.
The $d-\frac{d+1}{2}=\frac{d-1}{2}$ 
in-neighbours of $x$ that do not extend to a directed 3-cycle together with $(x,y)$ are also in-neighbours of $y$. 
Similarly, there are $\frac{d-1}{2}$ out-neighbours of $y$ 
that are also out-neighbours of~$x$. 
Then finally, the remaining 
$d-1-\frac{d-1}{2}=\frac{d-1}{2}$ out-neighbours of~$x$, 
other than these and $y$, must be in-neighbours of $y$.

Hence, there exist $\frac{(d+1)(d-1)}{4}$ 
directed 4-cycles that have $(x,y)$ as a diagonal.
Correspondingly, there are 
$\frac{(d+1)(d-1)}{8}\binom{n}{2}=\frac{3}{4}n\binom{d+1}{3}$ 
directed $4$-cycles, 
as each of them is counted twice by looking at its two diagonals.  
\end{proof}

For $n=2d+1$, let $T_n^\circlearrowright$ be the {\it clockwise tournament} 
on the vertex set $\ZZ_n$, where the arc $(x,y)$ exists if and only if $y-x \in \{1,2,\ldots,d\}$. 

\begin{lemma}\label{th:clockwise}
If $n=2d+1$, then 
$\vec{\kappa}_4(T_n^\circlearrowright) = n \binom{d+1}{3}$.
\end{lemma}

\begin{proof}[{\bf Proof}]
Let $L$ be the number of solutions to the linear equation $x+y+z+w=2d+1$ subject to the constraint $1\leq x,y,z,w\leq d$.
Fix $k=x+y$. If $k\leq d$, then there are $k-1$ choices for the pairs $(x,y)$. Moreover, as $m:=2d+1-k\geq d+1$, there are $k$ choices for the pairs $(z,w)$ such that $z+w=m$, from $m-d\leq z,w\leq d$. Thus, there are $k(k-1)$ solutions $(x,y,z,w)$.
Otherwise if $k\geq d+1$, the symmetric argument gives another $m(m-1)$ solutions. Therefore, $L=2\sum_{k=2}^d k(k-1) = 4\binom{d+1}{3}$. 
For each $j\in\ZZ_n$, there are exactly $L$ directed 4-cycles containing $j$ and thus, $\vec{\kappa}_4(T_n^\circlearrowright)= nL/4 = n\binom{d+1}{3}$.
\end{proof}

\begin{proof}[\bf{Proof of \cref{th:subdivision}}]
By \cref{th:equivalence}, it suffices to show that 
$(H_n,\alpha)$ is not norming for any $2$-edge-colouring $\alpha$. 
Suppose to the contrary that $(H_n,\alpha)$ is norming. 
As $\alpha$ is transitive, it defines an arc-transitive tournament $T_n=\vec{K}_n(\alpha)$.

Let $\beta$ be the $2$-edge-colouring of $H_n$ obtained by replacing each arc $(x,y)$ of the clockwise tournament $T_n^\circlearrowright$ with two edges $\{x,v_{\{x,y\}}\}$ and $\{y,v_{\{x,y\}}\}$ having colour $1$ and $0$, respectively.
Then the number of colour-alternating $2m$-cycles 
in $(H_n,\alpha)$ and $(H_n,\beta)$ 
is equal to the number of directed $m$-cycles 
in $T_n$ and $T_n^\circlearrowright$, respectively.
That is, $\kappa_{2m}(\alpha)=\vec{\kappa}_{m}(T_n)$ and $\kappa_{2m}(\beta)=\vec{\kappa}_m(T_n^{\circlearrowright})$.

Then \cref{lem:direct_K3,th:transitive_tournament,th:clockwise} give that 
\begin{align*}
    \kappa_6(\alpha) = \kappa_6(\beta) = \frac{1}{6} nd(d+1),~
    \kappa_8(\alpha) = \frac{3}{4} n \binom{d+1}{3},
    \text{ and }\kappa_8(\beta) = n \binom{d+1}{3}.
\end{align*}
The girth of $H_n$ is six, so \cref{eq:cycles} gives that 
$t_{H_n,\alpha}(1 + \varepsilon h_0) <
 t_{H_n,\beta }(1 + \varepsilon h_0)$
for $d>1$ and sufficiently small $\varepsilon>0$. 
Therefore, by \cref{th:P1}, $(H_n,\alpha)$ is not norming. 
\end{proof}

\section{Set-inclusion graphs and bipartite Kneser graphs}\label{sec:set-inclusion}

For a finite set $X$, let $\binom{X}{k}$ be the collection of all $k$-element subsets (or $k$-sets to be short) of~$X$.
The {\it set-inclusion graph} $I(n,k,r)$, $k\geq r$, is a bipartite graph between $k$-sets and $r$-sets of $[n]:=\{1,2,\ldots,n\}$ that indicates the inclusion.
That is, on the bipartition $\binom{[n]}{k}\cup\binom{[n]}{r}$, we place an edge $(X,Y)$ if the $k$-set $X$ contains the $r$-set $Y$.
For trivial examples, $I(n,n,k)$ and $I(n,k,0)$ are $K_{1,\binom{n}{k}}$ and $K_{\binom{n}{k},1}$, respectively, and $I(n,k,k)$ is a union of $\binom{n}{k}$ vertex-disjoint edges.
Because these examples are the only cases when $I(n,k,r)$ are weakly seminorming\footnote{A graph $H$ is weakly seminorming if $t_H(|\cdot|)^{1/{\rm e}(H)}$ defines a seminorm on $\F_\RR$ or (equivalently) $\F_\CC$. It is known~\cite{GHL19,L12} that disjoint unions of isomorphic stars are the examples that are not weakly norming.} but not weakly norming, we assume $n>k>r>0$ to avoid these exceptions throughout this section.

\medskip

Conlon and Lee~\cite[Theorem~1.1]{CL16} proved that all the set-inclusion graphs $I(n,k,r)$ are weakly norming, which includes a variety of examples:
$I(n,2,1)$ and $I(n,n-1,n-2)$ are isomorphic to the 1-subdivision of $K_n$ (depending on the orientation of the bipartition) discussed in the previous subsection;
$I(n,n-1,1)$ is isomorphic to $K_{n,n}$ minus a perfect matching, which is proven to be weakly norming by Lov\'asz~\cite[Theorem~14.2(c)]{L12} and not norming by Lee and Sch\"ulke~\cite[Theorem~1.6]{LS19};
more generallly, the particular case $I(n,n-r,r)$ that the graph is vertex-transitive 
is known as the {\it bipartite Kneser graph $H(n,r)$}. 

By~\cref{lem:eulerian}, norming graphs are Eulerian. 
This already precludes many set-inclusion graphs, e.g., $I(n,2t+1,1)$, from being norming.
However, there are extremely few Eulerian set-inclusion graphs determined to be norming or not.
The particular case $I(n,n-1,1)$ was essentially the only examples before we added $I(n,2,1)$ and $I(n,n-1,n-2)$ to the list in the previous subsection. 
In this subsection, 
we demonstrate that for the vast majority of parameters $(n,k,r)$, 
the set-inclusion graphs $I(n,k,r)$ are not norming.
In particular, we prove 
\begin{theorem}\label{th:Kneser}
$H(3,1)=C_6$ is the only bipartite Kneser graph that is norming. 
\end{theorem}

This generalises \cite[Theorem~1.6]{LS19}, where the case $r=1$ was solved.
We in fact rely on the result to simplify the case analysis, although we may reprove the result independently as roughly sketched in Concluding remarks.
\begin{theorem}[{\cite[Theorem~1.6]{LS19}}]\label{th:Kneser_1}
For $n>3$, the bipartite Kneser graph $H(n,1)$ is not norming.
\end{theorem}

Let $\alpha$ be a $2$-colouring of $I(n,k,r)$. 
For $A\in\binom{[n]}{k}$, denote by $\alpha_A$ 
the colouring of the edges of the complete $r$-graph 
on the vertex set $A$ 
defined by $\alpha_A(B) = \alpha(A,B)$. 
Let $\G_{\alpha,A}$ be the $r$-graph on $A$ where an $r$-set $B\subseteq A$ is an edge if and only if $\alpha(A,B) = 1$.

\begin{lemma}\label{th:set-inclusion}
Let $\alpha$ be a transitive colouring 
of the set-inclusion graph $I(n,k,r)$.
Then for any $k$-set $A \subset [n]$, 
the $r$-graph $\G_{\alpha,A}$ is edge-transitive 
and self-complementary . 
\end{lemma}

\begin{proof}[\bf{Proof}]
Let $B_1$ and $B_2$ be $r$-subsets of $A$ such that $\alpha(A,B_1)\neq \alpha(A,B_2)$.
Then there is an automorphism $\phi$ of $I(n,k,r)$ 
that maps $(A,B_1)$ to $(A,B_2)$ 
and reverses the colours of all edges, depending on the colours $\alpha(A,B_1)$ and $\alpha(A,B_2)$. 
In particular, $\phi$ fixes the vertex $A$ and hence, 
it induces an automorphism of the complete $r$-graph on $A$ which reverses the colours of all edges in $\alpha_A$ and maps $B_1$ to $B_2$. 
Then this automorphism is an isomorphism between $\G_{\alpha,A}$ and its complement and thus, $\G_{\alpha,A}$ is self-complementary.

By repeating the same argument with $B_1$ and $B_2$ that satisfies $\alpha(A,B_1)=\alpha(A,B_2)=1$, we obtain an isomorphism of $\G_{\alpha,A}$ that maps $B_1$ to $B_2$, i.e., $\G_{\alpha,A}$ is edge-transitive.
\end{proof}

The edge-transitive self-complementary $r$-graphs has been studied, for example in \cite{Chen18, Peisert01, Zhang92}, 
and their complete classification has been obtained 
under certain additional conditions.
The following result by Zhang~\cite{Zhang92} is about graphs.

\begin{theorem}[{\cite{Zhang92}}]\label{th:Zhang}
A vertex- and edge-transitive self-complementary $k$-vertex graph exists if and only if 
$k  \equiv 1 \pmod{4}$ and $k$ is a power of a prime.
\end{theorem}

Chen and Lu~\cite{Chen18} obtained an analogous result for $r$-graphs $\G$, $r\geq 3$. 

\begin{theorem}[{\cite[Theorem~1]{Chen18}}]\label{th:Chen}
Let $r \geq 3$ and $k \geq 2r$. 
An edge-transitive self-complementary $k$-vertex $r$-graph 
without isolated vertices 
exists if and only if 
$r=3$, $k \equiv 2 \pmod{4}$, and $k-1$ is a power of an odd prime.
\end{theorem}

In fact, the vertex-transitivity condition in \cref{th:Zhang} and the isolated-vertices-free condition in \cref{th:Chen} are not necessary. 
We give a short proof.

\begin{lemma}\label{th:self-compl2}
Any edge-transitive self-complementary graph $G$
is vertex-transitive.
\end{lemma}

\begin{proof}[\bf{Proof}]
Throughout the proof, by an \emph{automorphism} we mean that it is an automorphism of $G$ and hence of $\overline{G}$, too. 
We first claim that for any vertices $u$ and $v$ 
that have a common neighbour $w$ in~$G$, 
there exists an automorphism that maps $u$ to $v$. 
Indeed, there exists an automorphism~$\pi$ that maps $\{u,w\}$ to $\{v,w\}$. 
If $\pi(u)=v$ then we are done. Otherwise if $\pi(u)=w$, then $\pi(w)=v$, so~$\pi^2$ is an automorphism that maps $u$ to $v$. 
The same argument for $\overline{G}$ proves that if  $u$ and $v$ have a common neighbour in $\overline{G}$, 
then there exists an automorphism that maps $u$ to $v$. 

\medskip

Now consider the remaining case when 
$u$ and $v$ do not have a common neighbour in both~$G$  and~$\overline{G}$.
We may assume that $u$ and $v$ are not adjacent in $G$ 
(if they are, consider $\overline{G}$ instead of~$G$).
Let $U$ (resp. $V$) be the set of vertices adjacent to $u$ (resp. $v$) but not to $v$ (resp. $u$).
Then $V(G)=V(\overline{G})$ is partitioned into $\{u\},\{v\},U$, and $V$.

Suppose that there are vertices $u_1,u_2 \in U$ that are not adjacent in $G$. 
As $v$ is a common neighbour of $u$ and $u_1$ in $\overline{G}$, 
there is an automorphism $\pi_1$ that maps $u$ to $u_1$. 
Again, $u_2$ is a common neighbour of $u_1$ and $v$ in $\overline{G}$, 
so there is an automorphism $\pi_2$ that maps $u_1$ to $v$. 
Then~$\pi_1\pi_2$ maps~$u$ to~$v$. 
Hence, we may assume that $U\cup\{u\}$ is a clique in $G$, 
and similarly, $V\cup\{v\}$ is a clique in $G$, too. 
As the union of two disjoint cliques is not 
a self-complementary graph, 
there must be $u' \in U$ and $v' \in V$ that are adjacent in $G$. 
As $v$ is a common neighbour of $u$ and $u'$ in $\overline{G}$, 
there is an automorphism $\pi'$ that maps~$u$ to $u'$. 
Likewise, $v'$ is a common neighbour of $u'$ and $v$ in $G$, 
so there is an automorphism $\pi''$ that maps $u'$ to $v$. 
Then $\pi'\pi''$ maps $u$ to $v$. 
\end{proof}

\begin{lemma}\label{th:self-compl}
For $r \geq 2$, every edge-transitive self-complementary $r$-graph $\G$ has no isolated vertices.
\end{lemma}

\begin{proof}[\bf{Proof}]
Suppose to the contrary that that there is an isolated vertex $v$ in $\G$. 
Then there must be an isolated vertex $u\neq v$ in $\overline{\G}$. 
Let $X$ be an $r$-set that contains both $u$ and $v$. 
Then neither $\G$ nor $\overline{\G}$ contains $X$ as an edge, which is a contradiction.
\end{proof}

By mapping an $r$-edge in an $r$- graph $\G$ on $k$ vertices to its complement with $k-r$ elements, 
the $r$-graph $\G$ naturally corresponds to a $(k-r)$-graph.
Thus, an edge-transitive self-complementary $k$-vertex $r$-graph exists if and only if there exists an edge-transitive self-complementary $k$-vertex $(k-r)$-graph. We need this fact to close the gap in \cref{th:Chen} for $k < 2r$.
That is, if $k<2r$, an edge-transitive self-complementary $k$-vertex $r$-graph exists if and only if $k-r\leq 3$ and  the divisibility condition for $k$ is satisfied. 
When $r=1$, an edge-transitive self-complementary $1$-graph on $k$ vertices exists if and only if~$k$ is even. 

Let $\A$ denote the set of pairs $(k,r)$ with $k>r$ such that an edge-transitive self-complementary $k$-vertex $r$-graph exists. 
The next statement directly follows from combining \cref{th:Zhang,th:Chen,th:self-compl2,th:self-compl}.

\begin{theorem}\label{th:complete_hypergraph}
$(k,r) \in \A$ if and only if one of the following cases holds:
\begin{enumerate}[(1)]
\item
$r=1$ and \;$k$ is even;
\item
$r=2$, \;$k\equiv 1 \pmod{4}$, and $k$ is a power of a prime;
\item
$r=3$, \;$k\equiv 2 \pmod{4}$, and $k-1$ is a power of a prime;
\item
$r \geq 3$ is odd and $k=r+1$;
\item
$r \geq 7$, $r\equiv 3\pmod{4}$, $r+2$ is a power of a prime, and $k$ is either $r+2$ or $r+3$.
\end{enumerate}
\end{theorem}
Although this result already proves that the pairs $(k,r)\in\mathcal{A}$ are quite rare, 
we need more information about the symmetry of self-complementary edge-transitive graphs than this to prove~\cref{th:Kneser}.
In particular, a self-complementary edge-transitive $3$-graph has a uniform codegree of pairs.

\begin{lemma}\label{th:complete_123}
Let $r\in\{1,2,3\}$ and $|A| = k \geq 2r$. 
If $\G$ is an edge-transitive self-complementary $r$-graph on $k$ vertices, 
then $k-r$ is odd and each $(r-1)$-set is contained in exactly $(k-r+1)/2$ edges of $\G$. 
\end{lemma}

\begin{proof}[\bf{Proof}]
The case $r=1$ is trivial and the case $r=2$ follows from \cref{th:self-compl2}. Indeed, for $r=2$,~$\G$ is regular. As every vertex is incident to $k-1$ pairs in the complete graph, exactly half of which are in $\G$.
The case $r=3$ is proved in \cite{Chen18}; see~Remark~2 and Table~2 therein. 
\end{proof}

Switching from an $r$-graph to the corresponding $(k-r)$-graph on the same $k$-vertex set immediately proves the following:
\begin{corollary}\label{th:stars}
Let $r=2t-1$ and $k=r+s$ where $s\in\{1,2,3\}$. 
Let $\G$ be an edge-transitive self-complementary $k$-vertex $r$-graph. 
Then there are exactly $\binom{k}{s-1}=\binom{k}{2t}$ pairs $(X,Y)$ of vertex subsets in~$\G$ 
such that $|Y|=2t$, $X=\{x_1,x_2,\ldots,x_t\} \subset Y$, and $Y\setminus\{x_i\}$ is an edge in $H$ for $i=1,2,\ldots,t$. 
\end{corollary}

\cref{th:set-inclusion} and the fact that $I(n,k,r)$ is isomorphic to $I(n,n-r,n-k)$ (as a usual graph) give
\begin{corollary}\label{cor:kr}
Suppose that there exists a $2$-edge-colouring $\alpha$ of the set-inclusion graph $I(n,k,r)$ that is transitive. Then $(k,r) \in \A$ and $(n-r,n-k) \in \A$. 
\end{corollary}

Some of the set-inclusion graphs do have transitive colourings. For example, one may check that $I(6,4,1)$ has one.\footnote{The graph $I(6,4,1)$ is however not norming, as will be shown by~\cref{thm:nk1}.}
We hence do not aim to classify all the set-inclusion graphs with transitive colourings.
Even so, combining~\cref{cor:kr}~and~\cref{th:complete_hypergraph} with $k=n-r$ is still an attractive strategy to prove~\cref{th:Kneser}; it leaves us only a handful of cases to analyse.

\begin{corollary}\label{th:Kneser_list}
If the bipartite Kneser graph $H(n,r)$ 
admits a transitive colouring, 
then one of the following must be true: 
\begin{enumerate}[(1)]
\item
$r=1$ and \;$n$ is odd;
\item
$r=2$, \;$n\equiv 3 \pmod{4}$, and $n-2$ is a power of a prime;
\item
$r=3$, \;$n \equiv 1 \pmod{4}$, and $n-4$ is a power of a prime;
\item\label{it:Kneser_12}
$r \geq 3$ is odd and $n=2r+1$;
\item\label{it:Kneser_13}
$r \geq 7$, $r\equiv 3 \pmod{4}$, $r+2$ is a power of a prime, and $n$ is either $2r+2$ or $2r+3$.
\end{enumerate}
\end{corollary}
Our plan is to first prove that the set-inclusion graphs $I(n,k,r)$ are not norming for $r=2,3$ modulo extra assumptions on $n$ and $k$.
\begin{proposition}\label{prop:set_inclusion}
The set-inclusion graphs $I(n,k,r)$ is not norming if
\begin{enumerate}[(i)]
    \item $k\geq 4$ is even and $r=3$,
    \item $n\geq 7$, $k=5$, and $r=3$, or
    \item $k\geq 5$ is odd and $r=2$.
\end{enumerate}
\end{proposition}
This is unfortunately not enough to analyse all the cases in~\cref{th:Kneser_list}; in particular, \ref{it:Kneser_12} and \ref{it:Kneser_13} in~\cref{th:Kneser_list} are the cases when $r$ can be arbitrarily large but is close to $k=n-r$.
To conclude, we need an extra tool,~\cref{th:Kneser_3}, that relies on a slightly stronger estimate than Bertrand's postulate.
\begin{lemma}\label{th:prime}
For any integer $t \geq 2$, $t \neq 5$, 
there is an odd prime $p$ such that $\frac{3}{2}t \leq p < 2t$. 
\end{lemma}

\begin{proof}[\bf{Proof}]
If $2 \leq t \leq 4$, then $2t-1$ is a prime. 
The case $t \geq 6$ follows from the fact proven in~\cite{Nagura52} that, for any real $x\geq 9$, there is a prime $p$ such that $x \leq p < \frac{4}{3}x$. 
\end{proof}

\begin{lemma}\label{th:prime_divisor}
For any integer $t \geq 2$, 
there is an odd prime $p_t$ that divides $\binom{2t-1}{t}$ but does not divide 
$\binom{3t-1}{t-1}$ and $\binom{3t-1}{t}$. 
If $t\geq 4$ is even, then $p_t$ 
does not divide any integer from $[2t,3t+1]$. 
\end{lemma}

\begin{proof}[\bf{Proof}]
If $t=5$, then $3$ divides $\binom{9}{5}$, but not $\binom{14}{5}$ or $\binom{14}{4}$. 
If $t \neq 5$, then by \cref{th:prime}, 
there is an odd prime $p$ such that $\frac{3}{2}t \leq p < 2t$. 
Then $p$ divides $\frac{(2t-1)(2t-2)\cdots(t+1)}{(t-1)(t-2)\cdots 1}$ 
but does not divide any integer from $[2t,3t-1]$, 
and hence, does not divide 
$\binom{3t-1}{t-1}$ and $\binom{3t-1}{t}$. 
Furthermore, if $t \geq 4$ is even, 
then $\frac{3}{2}t$ is not a prime, 
so $p \neq \frac{3}{2}t$. Note also that $3t+1$ is odd.
Hence, both $3t$ and $3t+1$ are not divisible by $p$. 
\end{proof}

\begin{proposition}\label{th:Kneser_3}
Let $r=2t-1$ be an odd number. If
\begin{enumerate}[(i)]
    \item  $t\geq 2$ and $n=2r+1$, or
    \item  $t\geq 4$ is even, and $n$ is either $2r+2$ or $2r+3$,
\end{enumerate}
then the bipartite Kneser graph $H(n,r)$ has no transitive $2$-edge-colourings. 
\end{proposition}

\begin{proof}[\bf{Proof}]
Let $\alpha$ be a transitive $2$-edge-colouring of $H(n,r)$ and let $A$ be a $k$-set in $[n]$, $k=n-r$.
Then by~\cref{th:set-inclusion}, the $r$-graph $\G_{\alpha,A}$ on $A$ is self-complementary and edge-transitive.
Let $s:=k-r$ so that $s\in\{1,2,3\}$.
\cref{th:stars} then tells us that there are exactly $\binom{k}{s-1}$ pairs $(X,Y)$ of a $t$-set $X$ and a $2t$-set $Y$ such that 
$X\subseteq Y$ and $Y\setminus\{x\}$ is an edge in $\G_{\alpha,A}$ for every $x\in X$.

Let $\rho(X)$ be the number of pairs $(A,Y)$ where 
$X \subset Y \subseteq A \subset [n]$, $|A|=k$, $|Y|=2t$, 
and 
\begin{equation}\label{eq:stars}
  \alpha(A,Y\backslash\{x_1\}) \: = \:
  \alpha(A,Y\backslash\{x_2\}) \: = \:
  \ldots \: = \: 
  \alpha(A,Y\backslash\{x_t\}) . 
\end{equation}
That is, $(X,Y)$ is a pair described by~\cref{th:stars} with the choice $\G=\G_{\alpha,A}$ or its complement $\overline{\G}_{\alpha,A}=\G_{\overline{\alpha},A}$.
 
For a fixed $A$, there are exactly $2\binom{k}{s-1}$ choices of $Y \in\binom{A}{2t}$ and $X \in\binom{Y}{t}$ that satisfy \cref{eq:stars},
where the factor of $2$ comes from $\G_{\alpha,A}$ and $\overline{\G}_{\alpha,A}$.
Hence,
\begin{align}\label{eq:sum_X}
  \sum_{X \in \binom{[n]}{t}} \rho(X) \: = \:
  2\binom{k}{s-1} \binom{n}{k} \, .
\end{align}

Let the \emph{signature} of a subset $B$ of $[n]$ 
be the multiset of values $\rho(X)$ 
over all $t$-subsets $X \subseteq B$. 
Since~$\alpha$ is transitive, 
all $r$-sets have the same signature. 
Let $\rho_i$, $i=1,2,\ldots,m$, be all the distinct values in the signature of an $r$-set and let $\mu_i$ be the multiplicity of $\rho_i$ in the multiset.
Then
\begin{align*}
\sum_{i=1}^m \mu_i = \binom{r}{t}=\binom{2t-1}{t}~\text{ and }~
     \binom{n}{r} \sum_{i=1}^m \mu_i \rho_i \: = \:
  \binom{n-t}{t-1} \sum_{X \in\binom{[n]}{t}} \rho(X) \, .
\end{align*}
Indeed, the second equation comes from the sum of all elements in the signature of $B$ over all $B\in\binom{[n]}{r}$, where each $\rho(X)$, $|X|=t$, contributes $\binom{n-t}{r-t}=\binom{n-t}{t-1}$ times.
Combining this with~\eqref{eq:sum_X}, 
\begin{align*}
    \sum_{i=1}^m \mu_i \rho_i = 2 \binom{n-t}{t-1} \binom{k}{s-1}.
\end{align*}
In an $\ell$-set, $\ell\geq r$, there are $\binom{\ell}{r}$ $r$-subsets, each of which contains exactly $\mu_i$ $t$-sets $X$ with $\rho(X)=\rho_i$.
Thus, by double counting, the number of $t$-sets with $\rho(X)=\rho_i$ in each $\ell$-set is $\mu_i\binom{\ell}{r}/\binom{\ell-t}{r-t}$. 
In particular, every $(3t-1)$-element subset contains exactly $d_i := \mu_i \binom{3t-1}{t}/ \binom{2t-1}{t} $ 
$\:t$-subsets $X$ with $\rho(X) = \rho_i$. 
Hence, $d_i$ is an integer for each $i=1,2,\ldots,m$, so 
$d := \sum_{i=1}^m d_i \rho_i$ is also an integer. 

\medskip

On the other hand, 
\begin{equation*}
  d \: = \frac{\binom{3t-1}{t}} {\binom{2t-1}{t}} 
  \sum_{i=1}^m \mu_i \rho_i \: = \:
  \frac{2 \binom{n-t}{t-1} \binom{k}{s-1} 
  \binom{3t-1}{t}} {\binom{2t-1}{t}}.
\end{equation*}

\noindent(i) If $k=r+1=2t$ and $t \geq 2$, i.e., $s=k-r=1$ and $n=2r+1=4t-1$, then this simplifies to
\begin{align*}
    d=\frac{2 \binom{n-t}{t-1} \binom{3t-1}{t}} {\binom{2t-1}{t}}=\frac{2 \binom{3t-1}{t-1} \binom{3t-1}{t}} {\binom{2t-1}{t}}.
\end{align*}
By \cref{th:prime_divisor}, there is an odd prime factor $p$ of $\binom{2t-1}{t}$ that divides none of the integers $\binom{3t-1}{t-1}$ and $\binom{3t-1}{t}$. Thus, $d$ is not an integer.

\medskip

\noindent (ii) Suppose $s=k-r$ is either $2$ or $3$, and $t\geq 4$ is even. Then $n=2r+s=4t+s-2$.
Again by \cref{th:prime_divisor}, there is an odd prime factor $p$ of $\binom{2t-1}{t}$ that divides none of the integers from $2t$ to $3t+1$, as $t\geq 4$ is even.
This includes all the integers in the numerators of
$\binom{n-t}{t-1}=\binom{3t+s-2}{t-1}$, 
$\binom{k}{s-1}=\binom{2t+s-1}{s-1}$, and 
$\binom{3t-1}{t}$. Therefore, $d$ is not an integer. 
\end{proof}

We now give a proof of~\cref{prop:set_inclusion} by proving each case separately, although the method is almost the same to one another except small details.

\begin{proposition}\label{th:I_3}
If $k \geq 4$ is even, 
then $I(n,k,3)$ is not norming.
\end{proposition}

\begin{proof}[\bf{Proof}]
Suppose that $(I(n,k,3),\alpha)$ is norming for some  $2$-edge-colouring $\alpha$. 
By \cref{th:set-inclusion,th:complete_123}, 
for any $k$-set $A \subseteq [n]$, all distinct $x,y \in A$ extend to $(k-2)/2$ triples in $\G_{\alpha,A}$.
Thus, there are exactly $2\binom{(k-2)/2}{2}$ 
unordered pairs of triples $B,B' \subseteq A$ 
such that $x,y \in B,B'$ and 
$\alpha(A,B) = \alpha(A,B')$. 
Let $\LL$ be the collection of all such combinations $(A,\{B,B'\})$. That is, $A\in\binom{[n]}{k}$, $B,B' \in\binom{A}{3}$, $|B \cap B'|=2$, and $\alpha(A,B)=\alpha(A,B')$. 
Then $|\LL| = 2\binom{n}{k} \binom{k}{2} \binom{(k-2)/2}{2}$. 

For $(A,\{B,B'\})\in\LL$,~\cref{cor:manyC4} gives that
$\alpha(A',B)=\alpha(A',B')$ for all $k$-sets $A'$ that contain both~$B$ and $B'$. 
Since there are $\binom{n-4}{k-4}$ choices of such $A'$, the number of pairs $\{B,B'\}$ that extends to $(A',\{B,B'\})\in\LL$ is $|\LL| / \binom{n-4}{k-4}$. 
In other words, consider the auxiliary graph $H'$ on $\binom{[n]}{3}$ such that $B$ and $B'$ are adjacent if there exists $A$ with $(A,\{B,B'\})\in\LL$.
Then $\mathrm{e}(H')=|\LL|/\binom{n-4}{k-4}$.

By averaging, there exists $B_0=\{x,y,z\}$ such that
\begin{align*}
  \deg_{H'}(B_0)
   \geq \frac{2|\LL|}{\binom{n-4}{k-4}\binom{n}{3}} =\: \frac{4 \binom{n}{k} \binom{k}{2} \binom{(k-2)/2}{2} }{ 3\binom{n}{3} \binom{n-4}{k-4}} 
   = \: \frac{3(n-3)(k-4)}{2(k-3)}.
\end{align*}
Let $\ell=(k-2)/2$.
As each $B''$ adjacent to $B_0$ in $H'$ shares exactly two elements with $B_0$ and $\frac{(n-3)(k-4)}{2(k-3)}>\frac{k-4}{2}=\ell-1$, we may assume that there are at least $\ell$ distinct triples $B_1,\ldots,B_\ell$ that are adjacent to $B_0$ in $H'$ and contain $x$ and $y$.

Let $C=B_0\cup B_1\cup\ldots\cup B_\ell$. We may then write $C=\{x,y,z,z_1,z_2,\ldots,z_\ell\}$, where $B_i=\{x,y,z_i\}$, $1\leq i\leq \ell$.
Now let $A$ be a $k$-set that contains $C$. Then $(A,\{B_i,B_j\})\in\LL$ for any $0\leq i<j\leq\ell$, and thus, $\alpha(A,B_i)=\alpha(A,B_j)$.
That is, all $B_i$, $0\leq i\leq \ell$, are edges of a $3$-graph $\G$, where $\G$ is either $\G_{\alpha,A}$ or $\overline{\G}_{\alpha,A}$.
By \cref{th:set-inclusion}, $\G$ is self-complementary and edge transitive; however, the pair-degree of $\{x,y\}$ in $\G$ is at least $\ell+1=k/2$, which contradicts \cref{th:complete_123}.
\end{proof}

\begin{proposition}\label{th:I_5_3}
If $n\geq 7$, then $I(n,5,3)$ is not norming.
\end{proposition}

\begin{proof}[\bf{Proof}]
Suppose that $(I(n,5,3),\alpha)$ is norming for some  $2$-edge-colouring $\alpha$.
Let $x,y \in A$ be distinct. Then amongst the three triples in $A$ that contain $x$ and $y$, two triples are of the same colour, 
and the other triple is of the opposite colour. 
Let $\LL$ be the collection of all combinations $(A,\{B,B'\})$ such that $A\in\binom{[n]}{5}$, $B,B'\in\binom{A}{3}$,  
$|B \cap B'|=2$, 
and $\alpha(A,B)=\alpha(A,B')$. 
Then $|\LL| = 10 \binom{n}{5}$. 

By \cref{cor:manyC4}, if $(A,\{B,B'\})\in\LL$, then $\alpha(A',B)=\alpha(A',B')$ for all $A' \in\binom{[n]}{5}$ that contains $B \cup B'$. 
There are $n-4$ choices of such $A'$ given fixed $B,B'\in\binom{[n]}{3}$ with $|B\cap B'|=2$. 
Thus, the number of pairs $\{B,B'\}$ such that $\alpha(A',B)=\alpha(A',B')$ for all $5$-sets $A'$ containing $B \cup B'$ is $|\LL| / (n-4)$. 
Consider the auxiliary graph $H'$ on $\binom{[n]}{3}$ such that $B$ and $B'$ are adjacent if there exists $A$ with $(A,\{B,B'\})\in\LL$.
Then $\mathrm{e}(H')=|\LL|/(n-4)$.

By averaging, there exists $B_0=\{x,y,z\}$ such that
\begin{align*}
  \deg_{H'}(B_0)
   \geq \frac{2|\LL|}{(n-4)\binom{n}{3}} =\: \frac{20 \binom{n}{5} }{ (n-4)\binom{n}{3}} = n-3.
\end{align*}
As each $B''$ adjacent to $B_0$ in $H'$ shares exactly two elements with $B_0$ and $n>k\geq 6$, we may assume that there are at least two distinct triples $B_1$ and $B_2$ that are adjacent to $B_0$ in $H'$ and contain $x$ and $y$.

Let $A$ be a $5$-set that contains $B_0\cup B_1\cup B_2$.
As $\alpha(A,B_0)=\alpha(A,B_1)=\alpha(A,B_2)$ by definition of $\LL$, they are three distinct edges in $\G$, where $\G$ is either $\G_{\alpha,A}$ or $\overline{G}_{\alpha,A}$. By~\cref{th:set-inclusion} $\G$ is self-complementary and edge-transitive; however, $\{x,y\}$ has degree three in $\G$, which contradicts~\cref{th:complete_123}. 
\end{proof}

\begin{proposition}\label{th:I_2}
If $k \geq 5$ is odd, 
then $I(n,k,2)$ is not norming.
\end{proposition}

\begin{proof}[\bf{Proof}]
Suppose that $(I(n,k,2),\alpha)$ is norming for some  $2$-edge-colouring $\alpha$.
By \cref{th:set-inclusion}, 
for any pair $A \subset [n]$, 
$\G_{\alpha,A}$ and $\overline{\G}_{\alpha,A}$ is an edge-transitive and self-complementary $k$-vertex graph. Then by~\cref{th:complete_123}, every $a\in A$ is contained in $(k-1)/2$ edges of $\G_{\alpha,A}$ and of $\overline{\G}_{\alpha,A}$, respectively.
Thus, for each $a\in A$, there are exactly $2\binom{(k-1)/2}{2}$ unordered pairs $\{B,B'\}$ such that $B,B'\in\binom{A}{2}$, $|B\cap B'|=\{a\}$, and $\alpha(A,B)=\alpha(A,B')$. 
Let $\LL$ be the collection of all such combinations $(A,\{B,B'\})$. That is, $A\in\binom{[n]}{k}$, $B,B'\in\binom{A}{2}$,  
$|B \cap B'|=1$, 
and $\alpha(A,B)=\alpha(A,B')$. 
Then $|\LL| = 2k \binom{n}{k}\binom{(k-1)/2}{2}$. 

By \cref{cor:manyC4}, if $(A,\{B,B'\})\in\LL$, then $\alpha(A',B)=\alpha(A',B')$ for all $A' \in\binom{[n]}{k}$ that contains $B \cup B'$. 
There are $\binom{n-3}{k-3}$ choices of such $A'$ given fixed $B,B'\in\binom{[n]}{3}$ with $|B\cap B'|=1$. 
Thus, the number of pairs $\{B,B'\}$ such that $\alpha(A',B)=\alpha(A',B')$ for all $k$-sets $A'$ containing $B \cup B'$ is $|\LL| / \binom{n-3}{k-3}$. 
Let the auxiliary graph $H'$ on $\binom{[n]}{2}$ be such that $B$ and $B'$ are adjacent if there exists $A$ with $(A,\{B,B'\})\in\LL$.
Then $\mathrm{e}(H')=|\LL|/ \binom{n-3}{k-3}$.

By averaging, there exists $B_0=\{x,y\}$ such that
\begin{align*}
  \deg_{H'}(B_0)
   \geq \frac{2|\LL|}{\binom{n-3}{k-3}\binom{n}{2}} =\: \frac{4k \binom{n}{k}\binom{(k-1)/2}{2} }{ \binom{n-3}{k-3}\binom{n}{2}} = \frac{(n-2)(k-3)}{k-2}.
\end{align*}
As each $B''$ adjacent to $B_0$ in $H'$ shares exactly one element with $B_0$, we may assume that there are at least $\frac{(n-2)(k-3)}{2(k-2)}>\frac{k-3}{2}$ distinct pairs $B_1,B_2,\ldots,B_\ell$ that are adjacent to $B_0$ in $H'$ and contain $x$, where $\ell=\frac{k-1}{2}$.

Let $C=B_0\cup B_1\cup\ldots\cup B_\ell$. We may then write $C=\{x,y,z_1,z_2,\ldots,z_\ell\}$, where $B_i=\{x,z_i\}$, $1\leq i\leq \ell$.
Now let $A$ be a $k$-set that contains $C$. Then $(A,\{B_i,B_j\})\in\LL$ for any $0\leq i<j\leq\ell$, and thus, $\alpha(A,B_i)=\alpha(A,B_j)$.
That is, all $B_i$, $0\leq i\leq \ell$, are edges of a graph $\G$, where $\G$ is either $\G_{\alpha,A}$ or~$\overline{\G}_{\alpha,A}$.
By \cref{th:set-inclusion}, $\G$ is self-complementary and edge transitive. The degree of $x$ in $\G$ is at least $\ell+1=(k+1)/2$, which contradicts \cref{th:complete_123}.
\end{proof}

\begin{proof}[\bf{Proof of \cref{th:Kneser}}]
We analyse the cases listed in~\cref{th:Kneser_list} one by one.
\begin{enumerate}[(1)]
    \item \cref{th:Kneser_1} resolves all the cases when $r=1$.
    \item As $n$ is odd, $k=n-2$ is odd too. Thus,~\cref{prop:set_inclusion} (iii) covers all $k\geq 5$. \Cref{th:Kneser_3} then applies to the remaining case $k=3$ and $n=5$.
    \item As $n$ is odd, $k=n-3$ is even. Thus, \cref{prop:set_inclusion} (i) and (ii) settle all the cases.
    \item and \ref{it:Kneser_13} are resolved by \cref{prop:set_inclusion} (i) and (ii), respectively.
\end{enumerate}
Therefore,
$H(n,r)$ is not norming for $(n,r) \neq (3,1)$.
\end{proof}

We end this section with a sketch of our independent proof of~\cref{th:Kneser_1} and its slight generalisation,
which connects the case $g=4$ in~\cref{th:h3} to a cycle basis of a norming graph $H$ in an attractive way. For the definition of cycle spaces and basis, see, for instance,~\cite[Section 1.9]{Diestel}.
\begin{theorem}\label{thm:cycle_basis}
Let $(H,\alpha)$ be a norming $2$-coloured graph. If the $4$-cycles of $H$ generate the cycle space, then there exists a $2$-vertex-colouring $\beta$ such that $\alpha(uv)=\beta(u)+\beta(v)$ for every $uv\in E(H)$, where the addition is taken modulo $2$.
\end{theorem}
We omit the proof of this theorem.
It is not hard to see that the $4$-cycles in $I(n,k,1)$, $k\geq 4$, generate its cycle space, as the graph has diameter $2$. Then by using the $2$-vertex-colouring $\beta$ given by~\cref{thm:cycle_basis}, one may prove that there should exist a self-complementary edge-transitive $k$-graph (and an $(n-k)$-graph, too) on $n$ vertices, which contradicts~\cref{th:Chen} as $k>3$. 
This gives an alternative proof of~\cref{th:Kneser_1} and generalises it slightly. 
\begin{theorem}\label{thm:nk1}
For $n\geq k\geq 4$, $I(n,k,1)$ is not norming.
\end{theorem}

One might now wonder what the set-inclusion graphs $I(n,k,r)$ are yet undetermined to be norming or not.
Some case analysis gives that the unresolved cases for $I(n,r+s,r)$ fall inside $1\leq s\leq 3$, $r>6$, and $r+s+3<n<2r+s$. The smallest one is $I(13,8,7)$.

\section{Concluding remarks}

\noindent\textbf{Hypergraph generalisations.} 
The concept of graph norms naturally extends to $r$-graphs, as was noted by Hatami~\cite{H09} who studied generalisations of Gowers' octahedral norms~\cite{G06,G07}.
For precise definitions, we refer the reader to~\cite{CL16,H09}.
It would not be difficult to adjust our proofs to generalise~\cref{thm:equiv,thm:unique} for $r$-graphs.
\begin{theorem}\label{thm:hypergraph}
An $r$-graph $\mathcal{H}$ is real-norming if and only if it is complex-norming. 
Moreover, a $2$-edge-colouring $\alpha$ such that $(\mathcal{H},\alpha)$ is complex-norming is unique up to isomorphism.
\end{theorem}
\noindent It might be possible to obtain a hypergraph generalisation of~\cref{thm:equivalence_asymm} but is less straightforward.
For an $r$-graph $\mathcal{H}$, one should consider all the $r!$ permutations of the $r$-variables that require $r!$ colours instead of two. 
Proving such a variant of~\cref{th:monochrom_asymm} seems to present technical difficulties.

\medskip \vspace{3mm}

\noindent\textbf{Other examples of non-norming graphs.}
The ideas developed in~\cref{sec:cycles,sec:set-inclusion} also apply to other reflection graphs. 
For instance, an analogous method to the proof of~\cref{th:Q_2d,th:subdivision} also proves that the 1-subdivision $Q_d'$ of the $d$-dimensional hypercube $Q_d$, $d>2$, and the 1-subdivision $K'_{2;r}$ of the complete $r$-partite graph $K_{2,2,\ldots,2}$, $r>3$, with two vertices on each part are not norming. 
Indeed, $Q_d'$ and $K'_{2;r}$ are reflection graphs, since $Q_d'$ is the incidence graph of edges and vertices of the $d$-dimensional hypercube and $K'_{2;r}$ is isomorphic to the incidence graph of the $(r-1)$- and $(r-2)$-dimensional faces of the $r$-dimensional hypercube.
Moreover, $Q_2'$ and $K_{2;3}'$ are isomorphic to the 8-cycle and to the 1-subdivision of the octahedron, respectively, both of which are norming. Thus, we also obtain a complete classification of norming graphs in the two graph classes.

On the other hand, our proofs through~\cref{sec:cycles,sec:set-inclusion} rely on numerous ad-hoc techniques. 
It would be interesting to obtain a general statement that uses algebraic structure of reflection graphs to prove certain reflection graphs are not norming. The ultimate goal along this line may be answering~\cref{conj:norming} for all reflection graphs.
\begin{question}
What reflection graphs are not norming?
\end{question}

\medskip\vspace{2mm}

\noindent\textbf{Stable involutions in norming graphs.}
Our result is still not enough to prove that there exists a stable involution in a norming graph, which is equivalent to the positive graph conjecture for norming graphs.
However, we seem to have reached somewhere close; \cref{thm:equiv} and \cref{th:edge-transitive} tell us that there exists a partition $E_0\cup E_1$ of edges of a norming graph $H$ such that $E_0$ and $E_1$ are mapped to each other by an automorphism.
We believe that one of the colour-reversing automorphisms is a stable involution. 
It would already be interesting to prove that an involutary automorphism of a norming graph always exists.

\medskip\vspace{5mm}

\noindent\textbf{Acknowledgements.} 
The first author is supported by the research fund of Hanyang University (HY-202100000003086).
Part of this work was also carried out while the first author was supported by IMSS Research Fellowship and ERC Consolidator Grant PEPCo 724903.

We would like to thank Jan Hladk\'y for careful comments on early drafts of the paper. The first author is grateful to David Conlon for helpful discussions. We would also like to thank the anonymous referee, whose comments helped us to improve the presentation of this paper. 

\bibliographystyle{abbrv}
\bibliography{references}

\end{document}